%% file: paperNEW.tex
\documentclass[11pt]{article}
\usepackage{amssymb,amsthm,amsmath,mathtools}
\usepackage[colorlinks=true, allcolors=blue]{hyperref}
\usepackage{authblk}
\usepackage{bbm}
\usepackage{thmtools}
\usepackage{enumitem}
\usepackage{fullpage}
\usepackage{indentfirst}
\usepackage{comment}
\usepackage{subcaption}
\usepackage{dutchcal}
\usepackage{todonotes}

\usepackage{pgf,tikz}
\usetikzlibrary{patterns}
\usepackage[numeric,initials,nobysame,msc-links,abbrev]{amsrefs}

\newtheorem{theorem}{Theorem}

\newtheorem{lemma}{Lemma}
\newtheorem{proposition}{Proposition}

\newtheorem{question}{Question}
\newtheorem{definition}{Definition}
\newtheorem{remark}{Remark}

\newtheorem{claim}[question]{Claim}

\newcommand{\cA}{\ensuremath{\mathcal A}}
\newcommand{\cB}{\ensuremath{\mathcal B}}
\newcommand{\cC}{\ensuremath{\mathcal C}}
\newcommand{\cD}{\ensuremath{\mathcal D}}
\newcommand{\cE}{\ensuremath{\mathcal E}}
\newcommand{\cF}{\ensuremath{\mathcal F}}

\newcommand{\cH}{\ensuremath{\mathcal H}}

\newcommand{\cL}{\ensuremath{\mathcal L}}

\newcommand{\cR}{\ensuremath{\mathcal R}}
\newcommand{\cS}{\ensuremath{\mathcal S}}

\newcommand{\cV}{\ensuremath{\mathcal V}}

\newcommand{\bbE}{{\ensuremath{\mathbb E}} }

\newcommand{\bbL}{{\ensuremath{\mathbb L}} }

\newcommand{\bbN}{{\ensuremath{\mathbb N}} }

\newcommand{\bbP}{{\ensuremath{\mathbb P}} }
\newcommand{\bbQ}{{\ensuremath{\mathbb Q}} }

\newcommand{\bbZ}{{\ensuremath{\mathbb Z}} }

\newcommand{\eps}{\varepsilon}
\newcommand{\1}{\mathbbm1}

\newcommand{\rs}[1]{{\color{blue}{#1}}}
\newcommand{\rsr}[1]{{\color{red}{#1}}}

\makeatletter
\newcounter{maintheorem}

\newenvironment{subtheorems}
{%
  \refstepcounter{theorem}%
  \setcounter{maintheorem}{\value{theorem}}%
  \setcounter{theorem}{0}%
}
{%
  \setcounter{theorem}{\value{maintheorem}}%
}
\makeatother

\title{Near-critical percolation with sparse reinforcements}
\author[1]{Estevão Borel}
\author[2]{Marcos Sá}
\author[3]{Rémy Sanchis}
\author[4]{Roger W. C. Silva}
\affil[1]{\footnotesize  Departamento de Matemática, Universidade Federal de Minas Gerais,  Brazil \texttt{estevaofb@ufmg.br}}
\affil[2]{\footnotesize  Departamento de Matemática, Universidade Federal de Minas Gerais, Brazil \texttt{marcospy6@ufmg.br}}
\affil[3]{\footnotesize  Departamento de Matemática, Universidade Federal de Minas Gerais. Brazil \texttt{rsanchis@mat.ufmg.br}}
\affil[4]{\footnotesize Departamento de Estatística, Universidade Federal de Minas Gerais, Brazil \texttt{rogerwcs@est.ufmg.br}}

\date{\today}

\begin{document}

\maketitle

\begin{abstract} 


We study the effect of sparse random reinforcement on near-critical Bernoulli site percolation on the first quadrant of the square lattice. The reinforcement is generated by randomly selecting horizontal and vertical lines and increasing the opening probability $p$ only at their intersections, so that the reinforcement is supported on a sparse set of points with long-range dependence. We determine how the geometry of this random set influences the phase transition.

We first investigate whether the reinforcement lowers the critical threshold. When the spacings between consecutive selected lines are geometric, we prove that, for every $p$ above the critical threshold, percolation occurs even if the opening probability $q$ outside the reinforcement is strictly below criticality. In contrast, if the spacings in one direction decay slower than exponentially while those in the other are geometric, then no threshold shift occurs: percolation is impossible for $q$ below the critical threshold, regardless of $p$.

We next study the critical case $q=p_c$. We show that sufficiently high moments of the spacing distributions imply percolation at criticality, whereas the absence of a fractional moment of order $\alpha\in(0,1)$ rules it out. Both regimes remain unchanged after an additional independent thinning of the reinforced vertices. Thus,   reinforcement supported on random intersections may either lower the critical threshold or leave it unchanged while fundamentally altering the nature of the phase transition.

\end{abstract}

\noindent\textbf{MSC2020:} 82B43, 82B27
\\
\textbf{Keywords:} random environment, inhomogeneous percolation, critical percolation

\section{Introduction}
\subsection{Background and motivation}

Percolation theory provides one of the simplest mathematical frameworks for studying the emergence of large-scale connectivity in random media. In the classical Bernoulli bond percolation model on the square lattice $\mathbb{L}^2=(\bbZ^2,\cE(\bbZ^2))$, each edge $e$ is independently declared open or closed according to a Bernoulli random variable $\omega(e)$ with mean $p_e$. The collection of open edges (those such that $\omega(e)=1$) defines a random subgraph of $\mathbb{L}^2$, whose connected components are called open clusters. A central question in the theory is whether this random graph contains an infinite open cluster, in which case we say percolation occurs. 

The percolation model is called homogeneous when $p_e=p$ for every edge $e$, for some fixed $p\in[0,1]$. In this case, it is well known \cite{BH} that there exists a critical parameter $p_c^b(\bbL^2)\in(0,1)$ (here $b$ stands for bond) such that, for $p>p_c^b(\bbL^2)$, an infinite open cluster exists almost surely, whereas for $p< p_c^b(\bbL^2)$ all open clusters are finite almost surely.

One way to introduce inhomogeneities is by modifying the parameters $p_e$ for edges lying at a prescribed collection of horizontal and vertical lines. More precisely, for $\cC, \cR\subset \bbZ$, let
$$
\Lambda := (\mathcal{C} \times \mathbb{Z}) \cup (\mathbb{Z} \times \mathcal{R}).
$$

Given $x\in\cC$, let $\cE_x$ be the set of edges with both endpoints in $\{x\}\times\bbZ$, and write $\cE(\cC)=\cup_{x\in\cC}\cE_x$. Define $\cE(\cR)$ analogously and set 
\begin{equation}\label{edge_lambda}
\cE(\Lambda)=\cE(\cC)\cup \cE(\cR).
\end{equation}

For $p,q\in[0,1]$, write $\bbP^\Lambda_{p,q}$ for the probability measure on $\{0,1\}^{\cE(\bbZ^2)}$ under which the $\omega(e)$'s  are independent Bernoulli random variables with mean $p_e$, where
$$p_e = \begin{cases} p, & \text{if } e\in \cE(\Lambda), \\ q, & \text{if } e\notin \cE(\Lambda). \end{cases}$$

The study of such models is a central theme within percolation theory. For instance, in \cite{Z}, Zhang considers the model with $\cC=\{0\}$, $\cR=\emptyset$, and $q=p_c^b(\bbL^2)$, and shows that no percolation occurs for any $p<1$. On the other extreme, when $\cC=\bbZ$, $\cR=\emptyset$, we know by the work of Kesten (see page 54 in \cite{Kesten}) that percolation occurs whenever $p+q>1$.

 Since the seminal work of McCoy and Wu \cite{MW,MW2}, a recurring theme in the field has been understanding how local modifications of the medium affect the model phase transition. Of particular interest are situations in which a random subset of the lattice receives a preferential treatment, for instance by increasing the probability of an edge being open. 
 From the mathematical viewpoint, they provide a mechanism through which a lower-dimensional set may influence global connectivity properties.

One possibility is to create a disordered environment without any imposed geometric structure, as in the work of Klein \cite{kle94}, who studied the decay of connectivity probabilities in such settings using multiscale techniques.

 A recent example where randomness is introduced through geometric structures is the Brochette model \cite{Brochette}. In this model, entire columns of the two-dimensional lattice are reinforced, and the spacing between these lines follows a geometric distribution.  Specifically, $\cC$ is a random subset of $\bbZ$ formed by points chosen independently with probability $\rho>0$ and $\cR=\emptyset$. The authors show that, for every $\eps>0$, there exists $\delta=\delta(\rho,\eps)>0$ such that the model with $p=p_c^b(\bbL^2)+\eps$ and $q=p_c^b(\bbL^2)-\delta$ percolates for almost every environment $\Lambda$. This framework extends earlier work of Aizenman and Grimmett \cite{AG}, who established a similar result when the spacing of the reinforced lines is uniformly bounded. We refer the reader to \cite{CSS} for an extension of the results of \cite{Brochette} for slabs of $\bbZ^3$.

 In \cite{H}, Hoffman considers a family of models introducing randomness by stretching the lattice. Concretely, the environment is generated by retaining points in $\cC\subset\bbZ$ and $\cR\subset\bbZ$ according to independent Bernoulli variables with mean $\rho>0$.  The author shows a non-trivial phase transition in $p$ for this model when $q=0$. The framework and main result of \cite{H} will be revisited in greater detail in Section \ref{proof_theo_1}, as they play an important role in our arguments (see Lemma \ref{t:Hoffmann} in Section \ref{prop_1}).

 A common characteristic of these models is that, although the quenched measure $\bbP_{p,q}^{\Lambda}$ is typically inhomogeneous, the corresponding annealed measure — obtained by averaging $\bbP_{p,q}^{\Lambda}$ over all realizations of $\Lambda$ — is homogeneous in the sense that every edge has the same marginal probability of being open. Moreover, since the environment law is translation-invariant, the annealed measure is translation-invariant as well. Nevertheless, the resulting model is not independent: edges in the same vertical column remain positively correlated, and these correlations do not decay with the distance between them. Similar random-environment effects have been studied in a variety of settings, including percolation \cite{JJV,JLV,JMP,KSV}, the Ising model \cite{CK,CKP,DHP}, the contact process \cite{A,BDS,K,L,NV}, and the voter model \cite{F}.

 A broader mathematical question underlying these models is to understand how sparse defects or lower-dimensional random structures influence the global behavior of a critical system. Such mechanisms arise naturally in probability, statistical mechanics, and the study of random media, where a distinguished subset of the space may alter large-scale connectivity despite occupying a negligible fraction of the ambient lattice. From this perspective, the problem is not specific to percolation, but concerns the interplay between geometry, disorder, and phase transitions in stochastic lattice models. Understanding which geometric properties of the reinforcing set are relevant—and which are not—is therefore of independent mathematical interest.
 
  The present paper investigates a substantially sparser reinforcement mechanism for site percolation, in which vertices are declared open or closed according to Bernoulli random variables rather than edges. Rather than reinforcing entire rows or columns, we randomly select horizontal and vertical lines and reinforce only their intersections. The resulting environment then consists of isolated points arranged according to a random pattern. Since these reinforced vertices form a locally zero-dimensional set, it is far from obvious whether they can still alter the phase transition of the underlying percolation model. This reduction in the local dimension of the reinforcements brings the model near the limit of what may influence macroscopic connectivity.

 The main goal of this paper is to understand which features of a sparse random reinforcement determine how it affects the critical behavior of planar percolation. More precisely, we study a family of inhomogeneous site percolation models on the first quadrant of the standard nearest-neighbor square lattice $\bbL^2_+ = (\bbZ^2_+, \cE(\bbZ^2_+))$. The inhomogeneity arises from a random set of selected rows and columns, which are selected according to two independent random processes—one governing the rows and the other the columns. Vertices at the intersection of selected columns and rows are called reinforced, and we declare them open with probability $p_c+\eps$ for some $\eps>0$, while all other vertices are declared open with probability $p_c-\delta$ for some $\delta\geq 0$, where $p_c=p_c(\bbL^2_+)$ denotes the critical threshold for site percolation on $\bbL^2_+$. We investigate the phase diagram for this model, depending on the random mechanism generating the spacings between selected columns and rows.

\subsection{The model}

Let $(\xi_k^x)_{k \in \mathbb{Z}_+}$ and $(\xi_k^y)_{k \in \mathbb{Z}_+}$ be  independent sequences of  $i.i.d.$ non-negative random variables. These variables will represent the spacings between the selected columns and rows, respectively. Fixing the origin as a reference point, the set of selected columns $\mathcal{C} \subset \mathbb{Z}$ is given by
\begin{equation}\label{spe_col}\mathcal{C} := \{0\} \cup \left\{ \sum_{k=1}^{m} \xi_k^x : m \ge 1 \right\}.
\end{equation}
Analogously, the set of selected rows $\mathcal{R} \subset \mathbb{Z}$ is
\begin{equation}\label{spe_row}\mathcal{R} := \{0\} \cup \left\{ \sum_{k=1}^{m} \xi_k^y : m \ge 1 \right\}.
\end{equation}

\begin{figure}[ht]
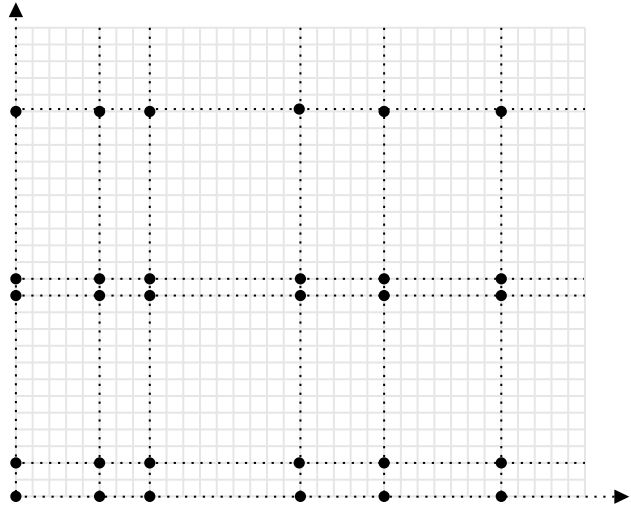

    \centering
    \input randomenv_tikz_2.tex
    \caption{An example of a random environment $\Xi\subset \bbZ^2_+$. Dashed lines are the selected columns and rows, while black dots are reinforced vertices.}
    \label{fig:randomenv}
\end{figure}

We write 
\begin{equation}\label{env_def}
\Xi= \mathcal{C}\times\mathcal{R}
\end{equation}
and call it the \textbf{random environment}. 
Throughout the paper, we denote the laws of $\cC$ and $\cR$ by $\mu_{\xi^x}$ and $\mu_{\xi^y}$, respectively. We also write $$\mu_{\xi}=\mu_{\xi^x}\times \mu_{\xi^y}.$$

We consider a site percolation model on $\mathbb{L}^2_+=(\bbZ^2_+,\cE(\bbZ^2_+))$. A percolation configuration is an element $\omega \in \{0, 1\}^{\mathbb{Z}^2_+}$, where $\omega(v) = 1$ (respectively, $\omega(v) = 0$) means that the vertex $v$ is \textbf{open} (respectively, \textbf{closed}). For a fixed environment $\Xi$ and $p,q\in[0,1]$, we define the product measure $\mathbb{P}_{p,q}^{\Xi}$ under which $(\omega(v))_{v\in\bbZ^2_+}$ are independent Bernoulli random variables with mean
\begin{equation}\label{perc_measure}\mathbb{P}_{p,q}^{\Xi} (\omega(v) = 1)= \begin{cases} p, & \text{if } v\in \Xi, \\ q, & \text{if } v\notin\Xi. \end{cases}
\end{equation}

Let $p_c$ be the critical threshold for homogeneous Bernoulli site percolation on $\bbL^2_+$.  Of special interest to us is the case $p=p_c+\eps$ and $q=p_c-\delta$, for some $\eps>0$ and $\delta\geq0$. In this scenario, a vertex $v\in\Xi$ is called \textbf{reinforced}.


Before we state our results, let us introduce some notation. We write $u \sim v$ if $u$ and $v$ are neighbors, that is, if $uv \in \mathcal{E}(\bbZ^2_+)$. A path in $\bbL^2_+$ is a sequence $x_0, x_1, \ldots, x_n$
of distinct vertices $x_i$ such that $x_{i-1}\sim x_{i}$, $i=1,\dots,n$.  A path is called \textbf{open} if all its vertices are open. We denote by $\{x_0 \longleftrightarrow x_n\}$ the event that $x_0$ is connected to $x_n$ by an open path. Finally, $\{x \longleftrightarrow \infty\}$ denotes the event that there exists an unbounded sequence $(x_n) \subset \mathbb{Z}_2^+$  such that $x$ is connected by an open path to each $x_n$.

\subsection{Results}

The central message of this paper is that the effect of sparse random reinforcement on criticality is determined by the tail behavior of the spacing distributions. Depending on this behavior, the reinforcement mechanism may either lower the critical threshold or leave the threshold unchanged while fundamentally altering the nature of the phase transition.

Given an environment $\Xi$ and $p=p_c+\varepsilon$, let
\begin{equation}\label{crit_q}
q_c(\varepsilon)=\inf\{q:\mathbb{P}_{p_c+\varepsilon,q}^{\Xi}(o\longleftrightarrow\infty)>0\}.
\end{equation}

When reinforcements are supported on entire columns, thus forming local one-dimensional structures, it is known that $q_c(\varepsilon)<p_c$ both when the spacing variables are geometric \cite{Brochette} and when they are heavy-tailed with sufficiently large tail index \cite{CSS}. This suggests that one-dimensional reinforcements robustly enhance connectivity, even under substantially different spacing distributions.

In contrast, our model reinforces only the intersections of the selected rows and columns, producing a random zero-dimensional reinforcement set. We show that this reduction in the dimension of the reinforcement fundamentally changes the picture. In this setting, whether the critical threshold is lowered depends crucially on the tail of the spacing distribution. While geometric spacings still yield $q_c(\varepsilon)<p_c$,  anisotropic environments with slower-than-exponential decay in one direction satisfy $q_c(\varepsilon)=p_c$.

The latter raises a second natural question: what happens at the critical value $q=p_c$ when the threshold is no longer shifted? Our results show that the tail of the spacing distribution continues to play a decisive role. Depending on the moment properties of the spacing distributions, the model may or may not percolate at criticality. Thus, although sufficiently sparse reinforcement may fail to lower the critical threshold, it can nevertheless create a discontinuity in the sense that, for a fixed $\varepsilon>0$,  if $q<q_c(\varepsilon)$, we have 
$$\mathbb{P}_{p_c+\varepsilon,q}^{\Xi}(o\longleftrightarrow\infty)=0,$$ whereas, if $q=q_c(\varepsilon)$, we have 
$$\mathbb{P}_{p_c+\varepsilon,q}^{\Xi}(o\longleftrightarrow\infty)>0.$$


A related question concerns the robustness of this reinforcement mechanism. Even if the reinforced vertices are capable of altering the phase transition, it is not clear whether this effect persists under further random perturbations of the environment. In particular, one may ask whether the reinforcement remains effective if the set of reinforced vertices is itself independently thinned. We return to this question in the final part of the paper; see Section \ref{sec_dil_env}.

\subsubsection{Shift of the critical threshold}

We first consider the case of geometric spacings. In this case, reinforced vertices appear sufficiently frequently to create a genuine shift of the critical threshold. We prove that for every $\varepsilon>0$ there exists $\delta>0$ such that percolation occurs at parameters $p=p_c+\varepsilon$ and $q=p_c-\delta$. Thus, despite the sparsity of the reinforcement set, the model possesses a supercritical phase strictly below the critical threshold of homogeneous percolation.

\begin{theorem}\label{diluted_geo} Let $(\xi_k^x)_{k \in \mathbb{Z}_+}$ and $(\xi_k^y)_{k \in \mathbb{Z}_+}$ be geometric with parameter $\rho\in(0,1]$. Then, for all $\varepsilon\in(0,1-p_c]$, there exists $\delta>0$ such that
$$\mathbb{P}_{p_c+\varepsilon,p_c-\delta}^{\Xi}(o\longleftrightarrow \infty)>0,\quad \mu_\xi \mbox{-almost every}\,\,\Xi.$$
\end{theorem}

Theorem \ref{diluted_geo} establishes that $q_c(\eps)<p_c$. We also show in Section \ref{sec_dil_env} that this phenomenon is robust under an additional independent dilution of the reinforcement set.



\subsubsection{Critical behavior without threshold shift}

Our next result shows that the threshold shift established in Theorem \ref{diluted_geo} is highly sensitive to the tail distribution of the reinforcement set. If the tail in one direction decays slower than exponentially, while the one in the other direction remains geometric, then arbitrarily large gaps occur frequently enough to destroy connectivity below $p_c$. 
This stands in contrast to the model studied in \cite{CSS}, where it is shown that the critical threshold remains shifted even when the spacings between consecutive selected columns have heavy-tailed distributions, provided all rows are deterministically present.

We say that the tail probabilities of a random variable $X$ decay slower than exponentially if 
\begin{equation}\label{heavy-tailed-2}
\limsup_{t\rightarrow \infty}\vartheta^{-t}\bbP(X>t)=\infty, \mbox{  for all $\vartheta\in(0,1)$.}
\end{equation}  
\begin{subtheorems}
\begin{theorem}\label{nopercolation} Let $(\xi_k^x)_{k \in \mathbb{Z}_+}$ be geometric with parameter $\rho\in[0,1]$, and $(\xi_k^y)_{k \in \mathbb{Z}_+}$ be such that \eqref{heavy-tailed-2} holds. Then, for every $\varepsilon\in(0,1-p_c]$ and $\delta\in(0,p_c]$, 
\begin{equation*}\bbP^{\Xi}_{p_c+\varepsilon,p_c-\delta}(o\longleftrightarrow\infty)=0, \quad \mu_{\xi}\mbox{-almost every }\Xi.
\end{equation*}
\end{theorem}

The proof of Theorem~\ref{nopercolation}, presented in Section~\ref{proof_theo_3}, extends beyond the setting considered above and also yields the following two variants.

\begin{theorem}\label{nopercolation_2}
Assume $\xi_1^x\stackrel{d}{=}\xi_1^y$. If $(\xi_k^x)_{k \in \mathbb{Z}_+}$ satisfies \eqref{heavy-tailed-2}, then, for every $\varepsilon\in(0,1-p_c]$ and $\delta\in(0,p_c]$, 
\begin{equation*}\bbP^{\Xi}_{p_c+\varepsilon,p_c-\delta}(o\longleftrightarrow\infty)=0, \quad \mu_{\xi}\mbox{-almost every }\Xi.
\end{equation*}
\end{theorem}

The assumption that $\xi_1^x\stackrel{d}{=}\xi_1^y$ in Theorem~\ref{nopercolation_2} can be removed at the expense of strengthening condition~\eqref{heavy-tailed-2}. More precisely,  we assume some regularity on the tail distributions of $(\xi^x_k)_{k\in\bbZ_+}$ and $(\xi^y_k)_{k\in\bbZ_+}$ by requiring that   
\addtocounter{equation}{+0}
\begin{equation}\label{heavy_tailed_3}
\tag{\theequation$'$}
\lim_{t\rightarrow \infty}\vartheta^{-t}\bbP(\xi^x_1>t)=\lim_{t\rightarrow \infty}\vartheta^{-t}\bbP(\xi^y_1>t)=\infty, \mbox{  for all $\vartheta\in(0,1)$.}
\end{equation}
\addtocounter{equation}{0}

\begin{theorem}\label{nopercolation_3} Let $(\xi_k^x)_{k \in \mathbb{Z}_+}$ and $(\xi_k^y)_{k \in \mathbb{Z}_+}$ be such that \eqref{heavy_tailed_3} holds. Then, for every $\varepsilon\in(0,1-p_c]$ and $\delta\in(0,p_c]$, 
\begin{equation*}\bbP^{\Xi}_{p_c+\varepsilon,p_c-\delta}(o\longleftrightarrow\infty)=0, \quad \mu_{\xi}\mbox{-almost every }\Xi.
\end{equation*}
\end{theorem}

\end{subtheorems}

At first sight, Theorems~\ref{nopercolation_2} and \ref{nopercolation_3} may appear to follow from Theorem~\ref{nopercolation}. However, it is unclear how to monotonically couple different spacing distributions and, even if this were possible, how such monotonicity in the environment would translate into monotonicity of the percolation measure.

Theorems~\ref{nopercolation}--\ref{nopercolation_3} show that, under a broad class of spacing distributions, sparse reinforcement does not lower the critical threshold, so that $q_c(\eps)=p_c$. This naturally raises the next question: can reinforcement still alter the phase transition at the critical point itself? We next identify moment conditions on the spacing random variables that determine whether percolation occurs at the critical value $q=p_c$. Roughly speaking, if the spacing distributions possess sufficiently high moments, then reinforcement, although too sparse to lower the critical threshold, is still capable of generating an infinite cluster at criticality. Conversely, if some fractional moment of order $\alpha\in(0,1)$ is infinite, then percolation does not occur at $q=p_c$. Consequently, under sufficiently high moment assumptions, the phase transition is discontinuous, in the sense that an infinite cluster already exists at the critical point. This behavior differs fundamentally from that of homogeneous planar percolation.

\begin{theorem}\label{diluted} There exists $\alpha_0>1$ such that, if $\bbE\left[(\xi_1^x)^\alpha\right]<\infty$ and $\bbE\left[(\xi_1^y)^\alpha\right]<\infty$ for some $\alpha>\alpha_0$, then for all $\varepsilon\in(0,1-p_c]$, we have
$$\mathbb{P}_{p_c+\varepsilon,p_c}^{\Xi} ( o\longleftrightarrow \infty) > 0, \quad \mu_{\xi}  \mbox{-almost every}\,\,\Xi.$$
\end{theorem}

The proof of Theorem~\ref{diluted} requires a different approach from that used for Theorem~\ref{diluted_geo}. Indeed, the latter relies on a comparison with a modified version of Hoffmann's model \cite{H}, whereas the corresponding model with heavy-tailed spacings does not percolate for any value of $p$ (see Theorem~1.2 in \cite{HSST2}).  Consequently, the comparison argument underlying the proof of Theorem~\ref{diluted_geo} is no longer available in the present setting. Instead, to prove Theorem \ref{diluted}, we develop a scale-by-scale renormalization scheme that avoids a full multiscale dynamical analysis.

The next theorem shows that the conclusion of Theorem~\ref{diluted} cannot be extended to arbitrary spacing distributions. Indeed, if the spacing distributions fail to possess a fractional moment of order $\alpha\in(0,1)$, then reinforcement is too sparse to sustain percolation at $q=p_c$, regardless of the value of $p>p_c$.

\begin{subtheorems}

\begin{theorem}\label{alpha_noperc}
 Assume $\xi_1^x\stackrel{d}{=}\xi_1^y$. If $\bbE\left[(\xi_1^x)^\alpha\right]=\infty$ for some $\alpha\in(0,1)$, then for all $\varepsilon\in(0,1-p_c]$, we have
$$\mathbb{P}_{p_c+\varepsilon,p_c}^{\Xi} ( o\longleftrightarrow \infty) = 0, \quad \mu_{\xi}  \mbox{-almost every}\,\,\Xi.$$
\end{theorem}

The assumption $\xi_1^x\stackrel{d}{=}\xi_1^y$ in Theorem~\ref{alpha_noperc} can be relaxed under a regular variation assumption on the tails of the spacing distributions. Recall that a function $H:(0,\infty)\rightarrow(0,\infty)$ is said to be regularly varying with tail index $\alpha>0$ if
\begin{equation}\label{regular}
\lim_{t\rightarrow\infty}\frac{H(xt)}{H(t)}=x^{-\alpha},
\end{equation}
for every $x>0$.

Under this assumption, we obtain the following extension of Theorem~\ref{alpha_noperc}.

\begin{theorem}\label{alpha_noperc_regular}
 Assume $\mathbb{P}(\xi_1^x>t)$ and $\mathbb{P}(\xi_1^y>t)$ are regularly varying functions with tail indices $0<\alpha_x,\alpha_y<1$. Then for all $\varepsilon\in(0,1-p_c]$, we have
$$\mathbb{P}_{p_c+\varepsilon,p_c}^{\Xi} ( o\longleftrightarrow \infty) = 0, \quad \mu_{\xi}  \mbox{-almost every}\,\,\Xi.$$
\end{theorem}

\end{subtheorems}

Together, Theorems~\ref{diluted} and~\ref{alpha_noperc} show that the occurrence of percolation at $q=p_c$ is governed, to a large extent, by the integrability properties of the spacing distributions. Theorem~\ref{alpha_noperc_regular} complements this picture by showing that the symmetry assumption $\xi_1^x\stackrel{d}{=}\xi_1^y$ in Theorem~\ref{alpha_noperc} can be replaced by a regular variation assumption on the tails of the spacing distributions. Our results leave open the intermediate range between finite high-order moments and infinite fractional moments, whose resolution would likely require a monotonicity principle with respect to the spacing law. 

The remainder of the paper is organized as follows. In Section~\ref{proof_theo_1}, we prove Theorem~\ref{diluted_geo} by combining a renormalization scheme with a comparison to the stretched-lattice percolation model introduced by Hoffmann~\cite{H}. In Section~\ref{proof_theo_3}, we prove Theorems~\ref{nopercolation}--\ref{nopercolation_3}. The argument is based on showing that the crossing probabilities of boxes at an appropriate length scale vanish as the box size tends to infinity. Section~\ref{proof_theo_2} is devoted to the proofs of Theorems~\ref{diluted}, \ref{alpha_noperc}, and \ref{alpha_noperc_regular}. The proof of Theorem~\ref{diluted} relies on a geometric analysis of the random environment together with a scale-by-scale renormalization argument, while those of Theorems~\ref{alpha_noperc} and \ref{alpha_noperc_regular} are based on the almost sure occurrence of infinitely many reinforcement-free quarter-annuli surrounding the origin.  Finally, Section~\ref{further} investigates the effect of an additional dilution of the environment on the phase transition and discusses possible extensions of our results.

\section{Geometric spacings - proof of Theorem \ref{diluted_geo}}\label{proof_theo_1}

\subsection{Outline of the proof}

The proof of Theorem \ref{diluted_geo} proceeds through a renormalization argument combined with a comparison to stretched-lattice percolation models introduced by Hoffmann~\cite{H}. We analyse boxes of side length $L$ and classify them as good or bad based on the geometry of the random environment inside a larger surrounding box. More precisely, a box is declared good if the selected rows and columns surrounding it are sufficiently dense, in the sense that consecutive selected lines are separated by at most logarithmic gaps. Standard estimates for geometric random variables imply that, for large $L$, a typical box is good with probability arbitrarily close to one.

We then define a renormalized dependent percolation process by declaring a box occupied whenever there exists an open circuit in the corresponding annulus. Using results from~\cite{Brochette}, we show that the occupation probability of a good box can be made arbitrarily close to one, whereas classical RSW estimates guarantee that even bad boxes remain occupied with uniformly positive probability. This produces a finite-range dependent percolation process with the properties described above.

The second part of the argument consists of comparing this renormalized model with a family of stretched-lattice percolation models. Through a sequence of stochastic domination arguments, we reduce the problem to a homogeneous site percolation process on a modified version of Hoffmann’s model. Hoffmann’s theorem, together with the Liggett--Schonmann--Stacey domination theorem \cite{LSS} and the Grimmett--Stacey \cite{GS} comparison between bond and site percolation, implies that this auxiliary model percolates for an appropriate choice of the parameters. Since the renormalized occupied process dominates this auxiliary model, we conclude that the original reinforced percolation model percolates.

\subsection{The block construction}\label{block_const}

Throughout the paper, we will need the notion of well-spaced intervals. We introduce a definition for future reference.

\begin{definition}[$r$-spaced intervals]
\label{def:k_spaced}
Let $a,b\in\bbZ_+$, $I = [a, b]$ an interval of integers, and $\mathcal{S} \subset \mathbb{Z}_+$ a set of points. Let $I\cap\mathcal{S}$ be ordered as $z_0 < z_1 < \dots < z_n$. We say that the interval $I$ is \textbf{$m$-spaced} with respect to $\mathcal{S}$ if
\[ \max_{j \in \{0, \dots, n-1\}} (z_{j+1} - z_j) \leq m. \] If the intervals $I_1$ and $I_2$ are both $m$-spaced with respect to $(\cS_1,\cS_2)$, $\cS_1\subset\bbZ_+$, $\cS_2\subset\bbZ_+$, respectively, we say that $A=I_1\times I_2$ is \textbf{doubly $m$-spaced} with respect to $(\cS_1,\cS_2)$.  
\end{definition}

Given $L\in \bbZ_+$ and $v=(v_1,v_2)\in \bbZ^2_+$, $v_1\geq 1,\,v_2\geq 1$, let
\begin{equation*}
B_L(v)=\left\{u=(u_1,u_2)\in\bbZ^2_+:0\leq u_i-Lv_i\leq L,\quad i=1,2\right\},
\end{equation*}
be a box of side length $L$.  Note that $\Pi=\{B_{3L}(v): v=(v_1,v_2)\in \bbZ_+^2\, v_1\geq 1, v_2\geq 1\}$ is a collection of overlapping boxes whose union is $\bbZ^2_+$.

\begin{definition}\label{good_box}Given an environment $\Xi=\cC\times\cR$, we say that the box $B_L(x)$ is \textbf{good} if $B_{3L}(v)$ is doubly $(c\log L)$-spaced with respect to $(\cC,\cR)$, for some constant $c>0$. Otherwise, we say $B_L(x)$ is \textbf{bad}.
\end{definition}

Since $\Xi$ is generated by geometric random variables with mean $\rho$, a standard calculation shows that
\begin{equation}\label{env_high}
\mu_\xi (B_L(v)\mbox{ is good})=\mu_\xi ([0,3L]\mbox{ is $(c\log L)$-spaced})^2\equiv \rho_L^2 \xrightarrow[L \to \infty]{} 1.
\end{equation}

Given $\Xi=\cC\times\cR$, we construct a $k$-dependent site percolation process on the graph with vertex set $\Pi$, induced by the model $\bbP^\Xi_{p_c+\varepsilon,p_c-\delta}$. Define $\cA_L(v)$ to be the event that there exists an open circuit in $B_{3L}(v)\setminus B_L(v)$, i.e., that there exists a path $v_0\sim v_1\sim\dots\sim v_k=v_0$ such that
\begin{enumerate}
\item $v_i$ belongs to $B_{3L}(v)\setminus B_L(v)$, for every $i=0,\dots,k$,
\item $\omega(v_i)=1$, for every $i=0,\dots,k$,
\item the winding number of the path around $v$ is non-zero.
\end{enumerate}

We say that the box $B_L(v)$ is \textbf{occupied} if $\cA_L(v)$ occurs; otherwise, we say that the box is \textbf{vacant}. It is clear that if the renormalized process percolates, then percolation also occurs in the microscopic model.

Write
$$\mathbb{P}^{\Xi}_{p_c+\eps,p_c-\delta} (B_L(v)\mbox{ is occupied}) = \begin{cases} p_G(\delta,L), & \text{if $B_{L}(v)$ is good} , \\ p_B(\delta,L), & \text{if $B_{L}(v)$ is bad}. \end{cases}$$

By classical RSW theory, there exists a constant $\sigma>0$ (which does not depend on $L$) such that $p_B(0,L)$ is at least $2\sigma$ whenever $B_L(v)$ is bad (we may consider the case where all vertices in $B_L(v)$ are open with probability $p_c$). Moreover, by Proposition 3 of \cite{Brochette}, if $B_L(v)$ is good, then $p_G(0,L)$ can be made arbitrarily close to one for $L$ sufficiently large. 

The block construction above yields a similar process to the original one, except that renormalized columns and renormalized rows are given according to $k$-dependent  Bernoulli random variables with marginal $\rho_L$, inducing an environment which we denote by $\Xi_L=(\cC_L\times \cR_L)$. In the corresponding (also $k$-dependent) percolation model on $\Xi_L$, vertices lying in the intersection of selected columns and selected rows are occupied with probability $p_G(0,L)$ and all the remaining ones are occupied with probability $p_B(0,L)$.
 
 Given $\Xi_L$ and $p,q\in[0,1]$, consider the $k$-dependent model above (with dependencies both in the environment and in the underlying percolation process) and denote its law by  $\bbP^{k,\Xi_L}_{p,q}$ ($k$ for $k$-dependent) with
$$\mathbb{P}_{p,q}^{k,\Xi_L} (\omega(v) = 1) = \begin{cases} p, & \text{if } v\in \Xi_L, \\ q, & \text{otherwise}. \end{cases}$$

For a renormalized vertex $v=(x,y)\in\bbZ^2_+$, note that $\mu_\xi(x\in \cC_L)=\mu_\xi(y\in \cR_L)=\rho_L$. The following proposition (whose proof is given in Section \ref{prop_1}) will help us to complete the proof of Theorem \ref{diluted_geo}.

\begin{proposition}\label{aux_res} For every $q>0$, there exists $\tilde\rho=\tilde\rho(q)$ and $\tilde p=\tilde p(q)$ large enough such that if $\rho_L\geq \tilde\rho$, then 
$$\bbP^{k,\Xi_L}_{\tilde p,q}(o\longleftrightarrow \infty)>0.$$

\end{proposition}

With Proposition \ref{aux_res}, we complete the proof of Theorem \ref{diluted_geo} as follows. Recall that $p_B(0,L)>2\sigma$ for every $L$. Take $q=\sigma$ and choose $L_0$ sufficiently large so that $\rho_{L_0}> \tilde\rho(\sigma)$ (which is possible by \eqref{env_high}) and $p_G(0,L_0)> \tilde p(\sigma)$. By continuity, we may find some $\delta>0$ such that $p_B(\delta,L_0)>\sigma$ and $p_G(\delta,L_0)>\tilde p(\sigma)$, completing the proof of Theorem \ref{diluted_geo}.

\subsection{Proof of Proposition \ref{aux_res}}\label{prop_1}

The proof of Proposition \ref{aux_res} relies on a result of Hoffmann \cite{H}, originally established for a bond percolation model on $\mathbb{Z}^2$. Our argument proceeds through a sequence of comparisons with intermediate models, requiring the introduction of additional notation. We ask the reader’s patience in what follows, as this notation is necessary to make the successive reductions precise.

We begin by formally introducing the percolation model on the stretched lattice and then state Hoffmann’s theorem.

\subsubsection{Bond percolation on a stretched lattice}

Let $\xi^x = (\xi_k^x)_{k \in \mathbb{Z}+}$ and $\xi^y = (\xi_k^y)_{k \in \mathbb{Z}+}$ be independent sequences of $i.i.d.$ geometric random variables with parameter $\rho\in(0,1]$. Let $\mathcal{R}$ and $\mathcal{C}$ denote the sets of selected columns and rows defined in \eqref{spe_col} and \eqref{spe_row}, respectively.

We define the environment $\Lambda \subseteq \mathbb{Z}_+^2$ by
\[
\Lambda := (\mathcal{C} \times \mathbb{Z}_+) \cup (\mathbb{Z}_+ \times \mathcal{R}),
\]
observing that the same product measure $\mu_{\xi^x} \times \mu_{\xi^y}$ governing the environment $\Xi$ also governs $\Lambda$. Whereas $\Xi$ consists of vertices that belong simultaneously to a selected row and a selected column, $\Lambda$ comprises all vertices that belong to some selected row or some selected column.

The environment $\Lambda$ induces a subgraph of the square lattice consisting of the selected rows and columns only. More precisely, we define the associated stretched lattice by
\[
\mathbb{L}_\Lambda = (\mathbb{Z}^2_+, \mathcal{E}(\Lambda)),
\]
where $\mathcal{E}(\Lambda$) is defined as in \eqref{edge_lambda}. We emphasize that this is not the induced subgraph in the usual graph-theoretic sense.

Given $\Lambda\subset \bbZ^2_+$, we construct a \textbf{bond} percolation model on $\mathbb{L}_\Lambda$ by assigning to each edge an independent Bernoulli random variable $\omega(e)$ with parameter $p \in [0,1]$. A percolation configuration is denoted by $\omega \in \{0,1\}^{\cE(\Lambda)}$, and we use the same classical terminology (e.g., open and closed edges).

Many vertices of the stretched lattice are isolated; consequently, they do not influence the connectivity events under consideration and may be discarded if convenient.

Given an environment $\Lambda$, we define the bond percolation measure $\mathbb{Q}_p^{b,\Lambda}$ by
\[
\mathbb{Q}_p^{b,\Lambda} \big( \omega(e) = 1 \big) = p=1-\bbQ^{b,\Lambda}_p(\omega(e)=0),
\]
for every $e\in\cE(\Lambda)$. Here $b$ stands for bond.

The following result was originally established by Hoffmann \cite{H}. We refer the reader to \cite{HSST2} for an alternative proof.

\begin{lemma}[Hoffmann]\label{t:Hoffmann} Let $(\xi_k^x)_{k \in \mathbb{Z}+}$ and $(\xi_k^y)_{k \in \mathbb{Z}+}$ be independent sequences of $i.i.d.$ geometric random variables with parameter $\rho\in(0,1]$. There exist $\rho^* < 1$ and $p^* < 1$ such that if $\rho \geq \rho^*$ and $p \geq p^*$, then
\[
\mathbb{Q}_p^{b,\Lambda}(o \longleftrightarrow \infty) > 0,
\quad \quad \mu_{\xi} \mbox{-almost every $\Lambda$.} \]
\end{lemma}

\subsubsection{The M-model}

In this section, we consider a variation of Hoffmann’s model. Given $\Lambda \subseteq \mathbb{Z}_+^2$, we construct a model by preserving those selected columns (rows) that are at unit distance from each other, while enlarging by a fixed constant factor the spacing between all remaining selected columns (rows). 

Let $(\xi_k^x)_{k \in \mathbb{Z}+}$ and $(\xi_k^y)_{k \in \mathbb{Z}+}$ be independent sequences of $i.i.d.$ geometric random variables with parameter $\rho\in(0,1]$. Let $M\in\bbZ_+$ be an arbitrary positive integer. Given an environment $\Lambda$ induced by $(\xi^x_k)_{k\in\bbZ_+}$ and $(\xi^y_k)_{k\in\bbZ_+}$, define
\[\mathcal{C}_M:=\{0\}\cup\left\{\sum_{k=1}^{m}M(\xi^x_k-1)+1:\;m \ge 1\right\},
\]
\[\mathcal{R}_M:=\{0\}\cup\left\{\sum_{k=1}^{m}M(\xi^y_k-1)+1:\;m \ge 1\right\},
\]
and write
\begin{equation}\label{environment_M}
\Lambda_M := (\mathcal{C}_M \times \mathbb{Z}_+) \cup (\mathbb{Z}_+ \times \mathcal{R}_M).
\end{equation}
In words, we keep unit spacing between selected rows and columns that are at distance one in $\Lambda$, while all other spacings are multiplied by $M$.
As before, we look at homogeneous \textbf{bond} percolation of parameter $p$ on the graph $(\bbZ^2_+,\cE({\Lambda_M}))$. We refer to this model as the \textbf{bond} $M$-model. 

\begin{lemma}\label{c:Hoffmann_crucial}
Let $M\in\bbZ_+$ be fixed and consider the \textbf{bond} $M$-model. Let $p^*$ be as in Lemma~\ref{t:Hoffmann}. 
There exists $\rho= \rho(M)<1$ such that
$$\mathbb{Q}_{p^*}^{b,\Lambda_M}(o \longleftrightarrow \infty)>0,$$
for $\mu_{\xi}$-almost every environment $\Lambda_M$.
\end{lemma}

\begin{proof}

We start with a simple probabilistic observation.
Let $X_1$ be a geometric random variable with parameter $r_1<1$ (denoted by $X_1\sim \mbox{Geo}(r_1)$). Also, let $X_2\sim\mbox{Geo}(r_2)$, with $r_2>1-(1-r_1)^M$, and define $$Y=M(X_1-1)+1.$$
A standard calculation shows that, for all $k\in\bbZ_+$, 
\begin{equation}\label{geo_dom}P(Y\geq k)\geq P(X_2\geq k).
\end{equation}

Let $$\xi^x=(\xi^x_k)_{k\in\bbZ_+} \mbox{ and } \xi^y=(\xi^y_k)_{k\in\bbZ_+}$$ be $i.i.d.$ geometric with parameter $\rho^*<1$ (as in the statement of Lemma \ref{t:Hoffmann}). Also, let $$\bar{\xi}^x=(\bar{\xi}^x_k)_{k\in\bbZ_+} \mbox{ and } \bar{\xi}^y=(\bar{\xi}^y_k)_{k\in\bbZ_+}$$ be $i.i.d.$ geometric with parameter $\rho_1>1-(1-\rho^*)^M$. By \eqref{geo_dom}, one can find a coupling of $(\xi^x,\xi^y)$ and $(\bar{\xi}^x,\bar{\xi}^y)$ such that
$$\bar{\xi}^x_k\leq \xi^x_k$$ and
$$\bar{\xi}^y_k\leq \xi^y_k,$$ for all $k\in\bbZ_+$. Since the percolation measure is monotone in $(\xi^x,\xi^y)$, the desired result follows from Lemma~\ref{t:Hoffmann}.
\end{proof}

\begin{remark}\label{gri_sta} By Grimmett-Stacey's theorem \cite{GS}, Lemma \ref{c:Hoffmann_crucial} still holds in the context of homogeneous site percolation with parameter $p_1=1-(1-p^*)^4<1$. For future reference, we state this fact as a lemma.
\end{remark}

Given $\Lambda$, we denote by $\bbQ_p^{s,\Lambda}$ the law of site percolation on $\Lambda$ with parameter $p$. Here $s$ stands for site.

\begin{lemma}\label{c:Hoffmann_crucial_site}
Let $M\in\bbZ_+$ be fixed and consider the homogeneous \textbf{site} $M$-model. Let $p^*<1$ be as in Lemma \ref{t:Hoffmann} and $p_1=1-(1-p^*)^4$ as in Remark \ref{gri_sta}. There exists $\rho = \rho(M)<1$ such that 
$$\mathbb{Q}_{p_1}^{s,\Lambda_M}(o\longleftrightarrow \infty)>0,$$ 
for $\mu_{\xi}$-almost every environment $\Lambda_M$.
\end{lemma}

Recall de definition of $\bbP^\Xi_{p,q}$ in \eqref{perc_measure}. The next lemma states that our original model percolates when $p$ and $\rho$ are sufficiently large.

\begin{lemma}\label{domination} Let $p_1$ be as in Lemma~\ref{c:Hoffmann_crucial_site} and $q>0$. There exists $\rho_1<1$ such that
$$\bbP_{p_1,q}^{\Xi}(o\longleftrightarrow\infty)>0,$$
for $\mu_{\xi}$-almost every environment $\Xi$.

\end{lemma}

\begin{proof} Given an environment $\Lambda$, its subset $\Xi=\cC\times\cR$, and $p,q\in[0,1]$, define the auxiliary measure
\begin{equation}\label{inhomo}\bar{\mathbb{P}}_{p,q}^{s,\Xi} (\omega(v) = 1) = \begin{cases} p, & \text{if } v\in \Xi, \\ q, & \text{if } v\in \Lambda\setminus\Xi, \\ 0, & \text{otherwise}.   \end{cases}
\end{equation}
Since $\bbP^{\Xi}_{p,q}$ stochastically dominates $\bar{\bbP}_{p,q}^{s,\Xi}$, it is enough to prove that the latter percolates. Choose $M\in\bbZ_+$ such that $p_1^{M}<q$ and apply Lemma~\ref{c:Hoffmann_crucial_site} to obtain $\rho_1=\rho_1(M)<1$ such that $\bbQ_{p_1}^{s,\Lambda_M}(o\longleftrightarrow\infty)>0$ for $\mu_{\xi}$-almost every $\Lambda_M$. 

By \eqref{environment_M}, if two consecutive vertices of $\Xi$ on a selected line are at $\Lambda$-distance $n$, then their counterparts in $\Xi_M\coloneq\cC_M\times \cR_M$ are at $\Lambda_M$-distance $M(n-1)+1$. Hence each vertex $v\in\Lambda\setminus \Xi$ corresponds naturally to a unique block $b(v)$ of $M$ consecutive vertices in $\Lambda_M\setminus \Xi_M$ (see Figure~\ref{fig:Mblock}). 

Given a configuration $\omega$ sampled according to $\bbQ_{p_1}^{s,\Lambda_M}$, define the map $\Psi:\{0,1\}^{\Lambda_M}\to\{0,1\}^{\Lambda}$,
\begin{equation}\label{omegaint}(\Psi\omega)(v) = \begin{cases} \omega(v), & \text{if } v\in \Xi, \\ \prod_{w\in b(v)}\omega(w), & \text{if } v\in \Lambda\setminus\Xi, \\ 0, & \text{otherwise}.   \end{cases}
\end{equation}

Since the blocks $b(v)$ are pairwise disjoint, the random variables $(\Psi\omega)(v)$ are independent Bernoulli with mean $p_1$ on $\Xi$. Also, since $p_1^M<q$ on $\Lambda\setminus \Xi$, the law of $\Psi\omega$ is stochastically dominated by $\bar{\mathbb{P}}_{p_1,q}^{s,\Xi}$. 

Any $\omega$-open path in $\Lambda_M$ projects naturally onto a path in $\Lambda$: every segment between consecutive vertices of $\Xi_M$ crosses all $M$ vertices of each block it traverses. Therefore, if the projected path visits a vertex $v\in \Lambda\setminus\Xi$, then every vertex of the corresponding block $b(v)$ is $\omega$-open, which implies $(\Psi\omega)(v)=1$. Consequently, the following inclusion holds

\begin{equation}\label{eq:contractinclusion}
    \{\omega\in\{0,1\}^{\Lambda_M}\ :\ o \longleftrightarrow\infty\text{ in } \omega\} \subseteq \Psi^{-1}(\{\omega'\in\{0,1\}^{\Lambda}: o\longleftrightarrow \infty\text{ in }\omega'\}).
\end{equation}
Taking probabilities on \eqref{eq:contractinclusion} under $\bbQ_{p_1}^{s,\Lambda_M}$ completes the proof.

\end{proof}
\begin{figure}[ht]
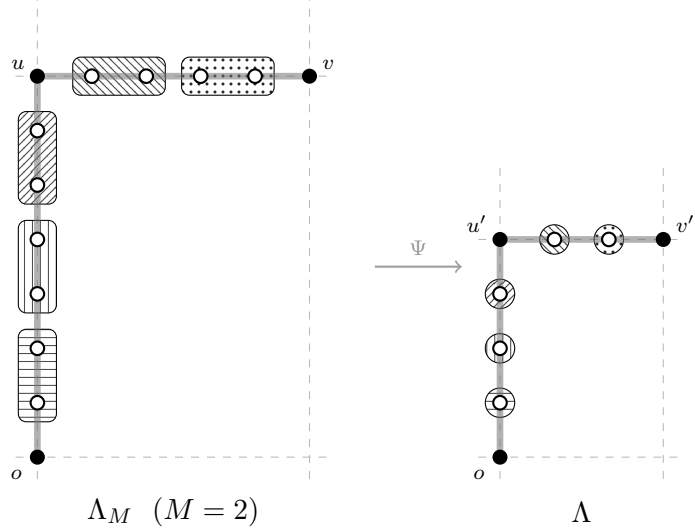

    \centering
    \input Mblock.tex
    \caption{The vertex identification given by Lemma~\ref{domination} on a sample path. Filled black dots are intersection vertices ($u,v\in\Xi_M$ and $u',v'\in\Xi$); white dots are non-intersection vertices on selected lines.}
    \label{fig:Mblock}
\end{figure}


Our construction in Section \ref{block_const} yields a $k$-dependent process such that the FKG inequality holds, in particular for bad boxes. Therefore, we are in precisely the same setting as in \cite{Brochette} (see the discussion after Lemma 8 in that paper). It follows from Lemma 9 in \cite{Brochette} that the process of occupied boxes stochastically dominates an independent percolation process, in the spirit of Liggett-Schonmann-Stacey's theorem \cite{LSS}. 

To conclude the proof of Proposition \ref{aux_res}, apply the theorem of Liggett, Schonmann and Stacey to find some $\tilde{p}<1$
such that any $k$-dependent family of random variables with marginal $\tilde p$ stochastically dominates Bernoulli independent variables with parameter $p_1$, where $p_1$ is as in Lemma~\ref{c:Hoffmann_crucial_site}. Applying the same theorem to the environment, we obtain $\tilde{\rho}<1$ such that the $k$-dependent environment (seen as Bernoulli 0-1 variables with marginal $\tilde{\rho}$) dominates the $i.i.d.$ environment with parameter $\rho_1$, where $\rho_1$ is as in Lemma \ref{domination}. Therefore, if $\rho_L>\tilde\rho$, Lemma~\ref{domination} implies that
$$\bbP_{\tilde p
,q}^{k,\Xi_L}(o\longleftrightarrow \infty)>0,$$
 completing the proof of Proposition \ref{aux_res} .

\section{No threshold shift - proof of Theorem \ref{nopercolation}}\label{proof_theo_3}

In this section, we prove Theorem \ref{nopercolation}. The proof follows the general strategy of Theorem 1.2 in \cite{HSST}. The key idea is that, with high probability, every sufficiently large rectangle contains a gap without reinforced vertices whose width is large. We show that the probability of crossing such rectangles therefore vanishes as their size tends to infinity.

 Recall that $\xi^x=(\xi_k^x)_{k\in\bbZ_+}$ is a sequence of $i.i.d.$   geometric variables with parameter $\rho\in(0,1]$, while  $\xi^y=(\xi_k^y)_{k\in\bbZ_+}$ is a sequence of $i.i.d.$ random variables satisfying \eqref{heavy-tailed-2}. Throughout this section, \bbQ denotes the joint law of the sequences $(\xi_k^x)_{k\in\bbZ_+}$ and $(\xi_k^y)_{k\in\bbZ_+}$, and $\bbE$ denotes expectation with respect to $\bbQ$.

Given a realization of $(\xi^x,\xi^y)$, let $\bbP_{p,q}^{\xi^x,\xi^y}$ denote the law governing the resulting percolation process in $\bbL^2_+$. Define the annealed measure
\begin{equation*}\mathbb{P}_{p,q}(\cdot)=\int \mathbb{P}_{p,q}^{\xi^x,\xi^y}(\cdot)\,d\bbQ,
\end{equation*}
viewed as a percolation measure on the sites of $\bbL^2_+$.

For $S\subset \bbZ^2_+$ and $A,B\subset S$, we write $\{A\stackrel{S}{\longleftrightarrow}B\}$ if some vertex in $A$ is connected in $S$ to some vertex in $B$. Finally, we denote by $m_\delta>0$ the exponential decay rate for subcritical site percolation on $\bbZ^2$ at parameter $p_c-\delta$.
\vspace{0.3cm}

\noindent\textit{Proof of Theorem \ref{nopercolation}}: Since increasing the opening probability at reinforced vertices can only facilitate percolation, it suffices to consider the extremal case in which every reinforced vertex is open deterministically. Thus, throughout the proof we work under $\mathbb{P}_{1,p_c-\delta}^{\Xi}$, with $\delta>0$.

Fix $T\in\bbZ_+$ and define the rectangle 
\begin{equation}\label{rectangle_beta}
R_T=[0,T]\times [0,T^\beta],
\end{equation}
with 
\begin{equation}\label{def_beta}
\beta=m_\delta/[4\log(1/\rho)],\, \rho\in(0,1].
\end{equation} 
Write $$L_1(T)=\{T\}\times[0,T^\beta]$$ and $$L_2(T)=[0,T]\times\{T^\beta\}$$ for the right and top sides of $R_T$, respectively. Since our percolation process lives in $\bbL^2_+$, we have
\begin{equation}\label{crossing}
\mathbb{P}_{1,p_c-\delta}(o\longleftrightarrow \infty)\leq \liminf_{T\rightarrow \infty}\left[\mathbb{P}_{1,p_c-\delta}\left(o\stackrel{R_T}{\longleftrightarrow}L_1(T) \right)+\mathbb{P}_{1,p_c-\delta}\left(o\stackrel{R_T}{\longleftrightarrow}L_2(T) \right)\right].
\end{equation}

Given the random pair $(\xi^x,\xi^y)$, we write
\begin{equation*}\label{kth_arrival}
S^x_k=\xi^x_1+\cdots+\xi^x_k,
\end{equation*}
and
\begin{equation}\label{max_arrival}
N^x(t)=\max\{k: S^x_k\leq t\}.
\end{equation}
The quantities $S^y_k$ and $N^y(t)$ are defined analogously.

Finally, define the events
\begin{equation*}
    H^T_k=\left\{\left[S^x_{k-1}, S^x_{k}\right]\times\left[0,T^\beta\right]\mbox{ is crossed horizontally}\right\}
\end{equation*}
and
\begin{equation*}
     V^T_k=\left\{\left[0,T\right]\times \left[S^y_{k-1}, S^y_{k}\right]\mbox{ is crossed vertically}\right\}.
\end{equation*}

We break the proof into two steps.
\bigskip

\noindent\textit{Step 1.} We first show that 
\begin{equation}\label{bound_1}
\lim_{T\rightarrow \infty}\mathbb{P}_{1,p_c-\delta}\left(o\stackrel{R_T}{\longleftrightarrow}L_1(T) \right)=0.
\end{equation}
Indeed, consider the event
$$ A_1=\left\{\max_{k\leq N^x(T)}  \xi^x_k>\frac{\log T}{2\log(1/\rho)}\right\}.$$
A standard estimate for the longest run of Bernoulli trials yields $$\bbQ(A_1)\longrightarrow 1, \mbox{ as $T$ goes to infinity.}$$

Suppose that $(\xi^x,\xi^y)\in A_1$, and let $k_1$ be such that $\xi^x_{k_1}>\frac{\log T}{2\log(1/\rho)}$. Then, every vertex $v=(v_1,v_2)$ with $S^x_{k_1-1}<v_1< S^x_{k_1}$ is open with probability $p_c-\delta$, and by standard percolation arguments we have, for every $v\in \{S^x_{k_1-1}\} \times [0,T^\beta]$, and for every $\delta>0$,  
\begin{equation}\label{exp_dec_2}
\mathbb{P}_{1,p_c-\delta}^{\xi^x,\xi^y}\left(v\stackrel{R_T}{\longleftrightarrow}\{S^x_{k_1}\}\times[0,T^\beta] \right)\leq T^{-\frac{m_\delta}{2\log(1/\rho)}}.
\end{equation}

Recalling \eqref{def_beta}, we then have 
\begin{align}
    \mathbb{P}_{1,p_c-\delta}\left(o\stackrel{R_T}{\longleftrightarrow}L_1(T) \right)&\leq \bbQ(A_1^c)+\int_{A_1}\mathbb{P}_{1,p_c-\delta}^{\xi^x,\xi^y}\left(o\stackrel{R_T}{\longleftrightarrow}L_1(T) \right)d\bbQ\nonumber\\
    &\leq \bbQ(A_1^c)+\int_{A_1}\mathbb{P}_{1,p_c-\delta}^{\xi^x,\xi^y}(H^T_{k_1})d\bbQ\nonumber\\
    & \stackrel{(\ref{exp_dec_2})}{\leq} \bbQ(A_1^c)+ T^\beta T^{-\frac{m_\delta}{2\log(1/\rho)}}\longrightarrow 0,\text{ as $T\longrightarrow\infty$}.
\end{align}

\bigskip

\noindent\textit{Step 2.} We now show that 
\begin{equation}\label{bound_4}
\liminf_{T\longrightarrow \infty}\mathbb{P}_{1,p_c-\delta}\left(o\stackrel{R_T}{\longleftrightarrow}L_2(T) \right)=0.
\end{equation}

Indeed, consider the event
$$A_2=\left\{\max_{k\leq N^y(T^\beta)}\xi^y_k>\frac{2}{m_\delta}\log (T)\right\}.$$
On the event $A_2^c$, every spacing before height $T^\beta$ is at most $\frac{2}{m_\delta}\log(T)$. Consequently,
$$T^\beta\leq S_{N^y(T^\beta)+1}\leq [N^y(T^\beta)+1]\frac{2}{m_\delta}\log(T),$$
which implies
\begin{equation}\label{inclusion}
N^y(T^\beta)\geq \frac{m_\delta T^\beta}{2\log T}-1\geq \frac{m_\delta T^\beta}{4\log T}.
\end{equation}
Therefore, 
\begin{align*}
\bbQ(A_2^c)&\leq \bigg[\bbP\bigg(\max_{k\leq (m_\delta T^\beta)/4\log(T)}\xi_k^y\leq (2/m_\delta)\log(T)\bigg)\bigg]\\
&\leq \bigg[1-\bbP(\xi^y_1> (2/m_\delta)\log(T))\bigg]^\frac{m_\delta T^\beta}{4\log(T)}\\
&\leq \exp\bigg(-\frac{m_\delta T^\beta}{4\log(T)}\bbP(\xi^y_1> (2/m_\delta)\log(T))\bigg).
\end{align*}
Setting $t=(2/m_\delta)\log(T)$ and $\vartheta =e^{-m_\delta\beta/4}$, condition \eqref{heavy-tailed-2} implies $$\liminf_{T\longrightarrow \infty} \bbQ(A_2^c)=0.$$

As before, assume that $\xi^x,\xi^y\in A_2$ are given and that $\xi^y_{k_2}>\frac{2}{m_{\delta}}\log(T)$ for some $k_2$. Then, every vertex $v=(v_1,v_2)$ with $S^y_{k_2-1}<v_2< S^y_{k_2}$ is open with probability $p_c-\delta$, and for every $v\in [0,T]\times\{S^y_{k_2-1}\}$, and for every $\delta>0$,
\begin{equation}\label{exp_dec_3}
\mathbb{P}_{1,p_c-\delta}^{\xi^x,\xi^y}\left(v\longleftrightarrow [0,T]\times\{S^y_{k_2}\}\right)\leq e^{-m_\delta\frac{2}{m_\delta}\log(T)}.
\end{equation}

Therefore,

\begin{align}
    \mathbb{P}_{1,p_c-\delta}\left(o\stackrel{R_T}{\longleftrightarrow}L_2(T) \right)&\leq \bbQ(A_2^c)+\int_{A_2}\mathbb{P}_{1,p_c-\delta}^{\xi^x,\xi^y}\left(o\stackrel{R_T}{\longleftrightarrow}L_2(T) \right)d\bbQ\nonumber\\
    &\leq \bbQ(A_2^c)+\int_{A_2}\mathbb{P}_{1,p_c-\delta}^{\xi^x,\xi^y}(V^T_{k_2})d\bbQ\nonumber\\
    &\stackrel{(\ref{exp_dec_3})}{\leq} \bbQ(A_2^c)+ T \cdot e^{-m_\delta\frac{2}{m_\delta}\log(T)}\nonumber,
\end{align}
yielding to
\begin{equation}\label{top}
\liminf_{T\longrightarrow \infty}\mathbb{P}_{1,p_c-\delta}\left(o\stackrel{R_T}{\longleftrightarrow}L_2(T) \right)=0.
\end{equation}
The theorem now follows from \eqref{crossing}, \eqref{bound_1}, and \eqref{bound_4}.

\qed

The proofs of Theorems~\ref{nopercolation_2} and~\ref{nopercolation_3} are immediate modifications of the preceding argument. We briefly indicate the required changes.

\bigskip

\noindent\textit{Proof of Theorem \ref{nopercolation_2}.} Since the spacing distributions in the two coordinate directions coincide, we choose $\beta=1$ in the definition of $R_T$ in \eqref{rectangle_beta}. By symmetry, 
$$\mathbb{P}_{1,p_c-\delta}\left(o\stackrel{R_T}{\longleftrightarrow}L_1(T) \right)=\mathbb{P}_{1,p_c-\delta}\left(o\stackrel{R_T}{\longleftrightarrow}L_2(T) \right).$$ 
Therefore, by \eqref{crossing},
$$\mathbb{P}_{1,p_c-\delta}(o\longleftrightarrow \infty)\leq 2\liminf_{T\rightarrow \infty}\left[\mathbb{P}_{1,p_c-\delta}\left(o\stackrel{R_T}{\longleftrightarrow}L_2(T) \right)\right].$$
Apply \eqref{top} to conclude the proof.

\qed
\bigskip

\noindent\textit{Proof of Theorem \ref{nopercolation_3}.} Again choose $\beta=1$. Since both spacing distributions satisfy the stronger assumption \eqref{heavy_tailed_3}, the arguments of Steps 1 and 2 apply in both coordinate directions. Consequently,
\begin{equation}\label{final}
 \lim_{T\rightarrow \infty}\left[\mathbb{P}_{1,p_c-\delta}\left(o\stackrel{R_T}{\longleftrightarrow}L_i(T) \right)\right]=0,\mbox{ $i=1,2$},
\end{equation}
and the conclusion follows from \eqref{crossing}.

\section{Critical behavior - proof of Theorem \ref{diluted}}\label{proof_theo_2}

Before proceeding with the proof of Theorem~\ref{diluted}, let us briefly explain why the strategy used in the proof of Theorem~\ref{diluted_geo} cannot be adapted to this case. In the case of geometric spacings, our approach relies on a block construction together with comparisons to a sequence of intermediate models, including a modified version of Hoffmann's model \cite{H}. However, the analogue of Hoffmann’s model with heavy-tailed spacings (instead of geometric) does not percolate for any value of $p$ (see Theorem 1.2 in \cite{HSST2}). Consequently, such a comparison becomes ineffective for establishing a phase transition in our model. 

Rather than adapting Hoffmann’s framework, we develop a scale-by-scale renormalization scheme that avoids a full multiscale dynamical analysis. At first glance, one drawback would be that this approach does not allow us to prove percolation when the non-reinforced vertices are open with probability strictly below $p_c$, regardless of the choice of $p$ in \eqref{perc_measure}. However, Theorem~\ref{nopercolation} shows that percolation indeed does not occur in this regime.

\subsection{Proof of Theorem \ref{diluted}}

The proof of Theorem \ref{diluted} combines a geometrical control of the random environment with a scale-by-scale renormalization argument. The first step consists of showing that, with positive probability, the reinforced vertices are sufficiently well distributed across all relevant scales. More precisely, if the spacing variables satisfy a polynomial tail bound with exponent $\alpha>1$, then for every $\beta>1/\alpha$ one can choose scales $L_k=2^kL_0$ so that, with positive probability, every box $[0,2L_k]^2$ is doubly $L_k^\beta$-spaced, in the sense of Definition \ref{def:k_spaced}. Conditioned on such environments, a result from \cite{Brochette} implies that crossing probabilities in the inhomogeneous model dominate those of homogeneous Bernoulli percolation with parameter $p_c+L_k^{-\beta c_2}$. Letting $c_1$ be an upper bound for the critical exponent of the correlation length, namely $\xi(p_c+\varepsilon)\leq C\varepsilon^{-c_1}$, and choosing $\beta$ so that $\beta c_1c_2<1$, one obtains, after a suitable adjustment of scales, crossing probabilities arbitrarily close to one.

The second step consists of propagating these crossing estimates across scales. Using a finite-size criterion due to Aizenman, Chayes, Chayes, Fr\"ohlich, and Russo \cite{aiz83}, we show that the probability of horizontal and vertical crossings converges to one super-exponentially fast along the sequence $L_k$. By concatenating these crossings through suitable rectangles and applying the FKG inequality, we construct an infinite open path with positive probability in every environment satisfying the spacing condition above. Finally, an ergodic argument shows that translates of such ``good'' environments occur almost surely somewhere in the plane, and since the origin is connected to one of these regions with positive probability, percolation follows.

\subsubsection{Controlling the environment}\label{env_control}

Throughout this section, the symbol \bbP denotes the joint probability measure of the sequences $(\xi_k^x)_{k\in\bbZ_+}$ and $(\xi_k^y)_{k\in\bbZ_+}$, while $\bbE$ denotes expectation with respect to $\bbP$.

We begin by controlling the environment's geometry. The key property we need is that, within any relevant region, the reinforced vertices are typically not too far apart. 

Recall the definition of $\Xi$ in \eqref{env_def}.  For a given $L_0$ and $k\geq 1$, consider the sequence $L_k\equiv 2^kL_0$, and a contraction parameter $\beta>0$. Recall Definition \ref{def:k_spaced} and consider the event
\[ \Psi_{\beta} = \left\{ \text{for all } k \ge 1, \text{ the box } [0, 2L_k]^2 \text{ is doubly $L_k^{\beta}$- spaced with respect to $\Xi$} \right\}.\]

In the following, recall the notation $\mu_{\xi}=\mu_{\xi^x}\times\mu_{\xi^y}$.

\begin{proposition}[Control of the environment]\label{prop:env_control}
Assume $(\xi_k^x)_{k \in \mathbb{Z}_+}$ and $(\xi_k^y)_{k \in \mathbb{Z}_+}$ as in the statement of Theorem \ref{diluted}, and let $\beta>1/\alpha$. Then, there exists $L_0$ sufficiently large such that $$\mu_{\xi}(\Psi_{\beta})\geq 1/2.$$
\end{proposition}
\begin{proof}
Given $\Xi=(\cC\times\cR)$, let $L\in\mathbb{Z}_+$, and write
\begin{equation}\label{column_space}
\cA_\beta =\{[0,2L]\text{ is }L^\beta\text{- spaced with respect to }\mathcal{C}\},
\end{equation}
\begin{equation}\label{row_space}
\cB_\beta =\{[0,2L]\text{ is }L^\beta\text{- spaced with respect to }\mathcal{R}\}.
\end{equation}
Recall that $S_n^x=\sum_{k\leq n}\xi_k^x$, $S_0^x=0$, and $N^x(t)=\max\{k\ge 0:\ S^x_k\}$.  
 For $L\in\bbZ_+$ and $r>0$ define
\[
Z_L(r)=\sum_{i\ge 0}\mathbf 1\{S^x_i\le 2L,\ \xi_{i+1}^x>r\},
\]
which is an upper bound for the number of gaps larger than $r$ inside $[0,2L]$. Clearly, if $[0,2L]$ is not $r$-spaced then $Z_L(r)\ge 1$, and hence by Markov's inequality
\[
\mu_{\xi^x}([0,2L]\text{ is not $r$-spaced})
\le \bbE[Z_L(r)],
\]
where $\bbE$ denotes expectation with respect to $\bbP$. Since $S^x_{i}$ and $\xi_{i+1}^x$ are independent, we have

\begin{align*}
\bbE[Z_L(r)]
&= \sum_{i\ge 0}\bbP(S^x_i\le 2L,\ \xi^x_{i+1}>r)\\
& = \sum_{i\ge 0}\bbP(S^x_{i}\le 2L)\,\bbP(\xi^x_{i+1}>r)\\
&= \bbP(\xi^x_1>r)\,\bbE[N^x(2L)+1].
\end{align*}

By definition, we have $N^x(L)\leq L$ and hence $\bbE[N^x(2L)+1]\le 3L$. Setting $r=L^\beta$, apply the Markov inequality again to obtain
\begin{equation*}
\mu_{\xi^x}((\cA_\beta)^c)
\le \bbE[Z_L(L^\beta)]
= 3c L^{1-\alpha\beta},
\end{equation*}
for some constant $c>0$. The same argument gives an identical bound for $\mu_{\xi^y}((\cB_\beta)^c)$, implying, by independence, that
\begin{equation}\label{spaced_bound}
\mu_{\xi}(\cA_\beta\cap \cB_\beta)\geq 1-6cL^{1-\alpha\beta}.
\end{equation}

From now on, take $L=L_0$ and $L_k=2^kL_0$, $k\geq 1$. Our goal is to show that, with high probability, all boxes $[0,2L_k]^2$ are doubly $L_k^{2\beta}$ - spaced with respect to $\Xi\in\Psi_\beta$. Define the event
\[
\Psi_{\beta,k} = \{[0,2L_k]^2 \text{ is doubly } L_k^\beta\text{- spaced with respect to $\Xi$}\},
\]
observing that \eqref{spaced_bound} gives $\mu_\xi(\Psi_{\beta,0})\geq 1-6L^{1-\alpha\beta}.$ It follows that
$$\mu_\xi((\Psi_{\beta,k})^c)\leq 6cL_0^{1-\alpha\beta}2^{k(1-\alpha\beta)}.$$

Since $\beta>1/\alpha$, the exponent $1-\alpha\beta$ is negative, so the series $\sum_k\mu_{\xi}\left(\left(\Psi_{\beta,k}\right)^c\right)$ is summable. Choosing $L_0$ large enough, we obtain

\begin{equation*}
\mu_{\xi}(\Psi_{\beta}^c)
\le \sum_{k\ge 1}\mu_{\xi}\left(\left(\Psi_{\beta,k}\right)^c\right)
< \tfrac12.
\end{equation*}
\end{proof}

\subsubsection{Technical lemmas}

We start this section with a classical lemma relating an upper bound on the two-dimensional correlation length with crossing probabilities. Consider the crossing events 
$$\cH_L=\{\mbox{horizontal crossing of $[0,2L]\times [0,L]$}\}$$
and 
$$\cV_L=\{\mbox{vertical crossing of $[0,L]\times [0,2L]$}\}.$$

Denote by $\bbP_p$ the law of independent Bernoulli percolation with parameter $p$ and $p_c$ its critical threshold. For a given $\gamma > 0 \mbox{ and } p > p_c$,  the correlation length is defined by
\begin{equation*}\label{cl_def} \xi_\gamma(p) = \inf\{L\geq 1: \bbP_p(\mathcal{H}_L)\geq 1-\gamma\}. 
\end{equation*} 
A fundamental result in the scaling theory of two-dimensional percolation is that, for $p = p_c+\varepsilon$, we have
\begin{equation}\label{crit_exp}\xi_{\gamma}(p_c+\varepsilon) \leq C_\gamma\varepsilon^{-c_1}, \quad \text{as } p \downarrow p_c,
\end{equation}
for some constants $c_1> 0$ and $C_\gamma>0$; see \cite{Kesten1987}.  Hence,


\begin{lemma}
\label{prop:supercritical_crossing}
Let $c_1> 0$ be as in \eqref{crit_exp}. For every $\gamma>0$, the following holds. For every $\eps>0$ and $p = p_c + \varepsilon$,  if $L$ satisfies
\[ L > C_\gamma\varepsilon^{-c_1}, \]
then 
\[ \mathbb{P}_p(\mathcal{H}_L) > 1-\gamma. \]
\end{lemma}

The same result holds for the event $\cV_L$. The proof of the next lemma follows the same strategy as the proof of Proposition 2 in \cite{Brochette}. It is based on a differential inequality derived from Russo's formula together with a local surgery argument around the nearest reinforced site. The adaptation to the present setting of site percolation with zero-dimensional reinforcements is straightforward. Although the geometric estimates involved in the surgery lead to slightly different polynomial exponents, the overall structure of the proof, the use of RSW estimates, and the final conclusion remain unchanged.


\begin{lemma}[Comparison with homogeneous percolation]
\label{lem:comparison_homogeneous}
Let $\varepsilon\in (0,1/2)$. There exists $c_2 > 0$ such that, for $r \in \mathbb{Z}_+$ sufficiently large and every $L \ge r$, the following holds. For every environment $\Xi$ such that $[0,2L]^2$ is doubly $r$-spaced, it holds that 
$$
    \mathbb{P}_{p_c+\varepsilon,p_c}^{\Xi}(\mathcal{H}_L)
    \;\ge\;
    \mathbb{P}_{\,p_c + r^{-c_2}}(\mathcal{H}_L),
$$
and 
$$
    \mathbb{P}_{p_c+\varepsilon,p_c}^{\Xi}(\mathcal{V}_L)
    \;\ge\;
    \mathbb{P}_{\,p_c + r^{-c_2}}(\mathcal{V}_L).
$$
\end{lemma}

\begin{figure}[ht]
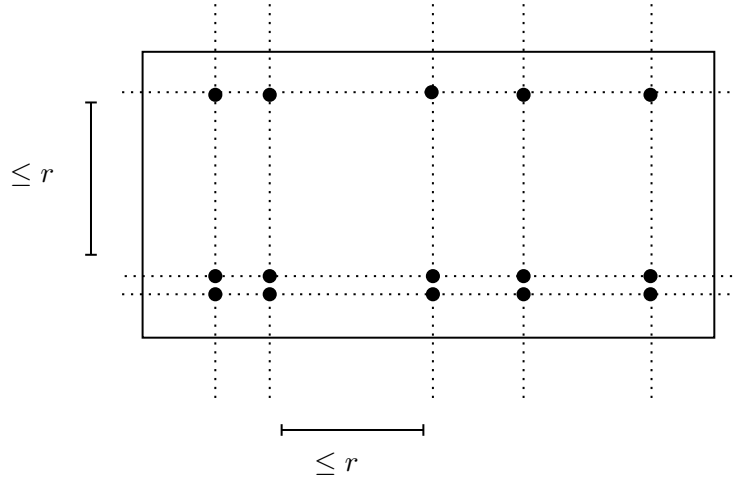

    \centering
    \input rspacedbox_tikz.tex
    \caption{$r$-spaced boxes.}
    \label{fig:rspacedbox}
\end{figure}

The last ingredient in our proof is a classical finite-size criterion due to Aizenman, Chayes, Chayes, Fröhlich, and Russo \cite{aiz83}. One of its consequences is the characterization of the supercritical phase of two-dimensional percolation in terms of crossing probabilities of large squares; see \cite{cha96}. Although originally formulated for bond percolation, the argument relies only on the FKG inequality, RSW estimates, and standard constructions based on overlapping crossings, and therefore applies equally to Bernoulli site percolation on $\bbZ_+^2$; see Comment 1 in the Introduction of \cite{KST}.

\begin{lemma}\label{lem:accfr}
Consider independent Bernoulli percolation on $\mathbb{Z}_+^2$ under the measure $\mathbb{P}_p$ and let $\sigma< 1$.  
If
\[
    \mathbb{P}_{p}(\mathcal{H}_L) \ge 1 - \tfrac{1}{16}\sigma,
\]
then for every $m \ge 1$,
\[
    \mathbb{P}_{p}(\mathcal{H}_{2^m L}) \ge 1 - \tfrac{1}{16}\sigma^{2^m}.
\]
\end{lemma}

The same result holds for the event $\cV_L$.

\subsubsection{Building the cluster}\label{build_cluster}

\noindent\textit{Proof of Theorem \ref{diluted}:} Consider the constants $c_1$ from \eqref{crit_exp} and $c_2$ from Lemma \ref{lem:comparison_homogeneous}, and let $\alpha> c_1c_2$ be fixed as in the statement of the theorem. Let $\beta$ be such that
\begin{equation}\label{beta}\frac{1}{\alpha}<\beta<\frac{1}{c_1 c_2}.\end{equation}

Define $L_k=2^kL_0$, where $L_0$ is chosen large enough so that Proposition \ref{prop:env_control} holds. For any environment $\Xi\in\Psi_{\beta}$, we have that for every $k\geq 1$ the box $[0,2L_k]^2$ is $L_k^{\beta}$-spaced.
Let us show that, for every $k\geq 1$, the probability of the event $\mathcal{H}_{L_k}$ increases to one exponentially fast in $L_k$. 

Let $\varepsilon>0$ and fix $\Xi\in\Psi_{\beta}$. In this case, Lemma \ref{lem:comparison_homogeneous} gives
\begin{equation}\label{env_bound}
\mathbb{P}_{p_c+\varepsilon,p_c}^{\Xi}(\mathcal{H}_{L_k}) \ge \mathbb{P}_{p_c + \varepsilon_k}(\mathcal{H}_{L_k}), \quad \quad \mbox{for all $k\geq 0$,}
\end{equation}
with $\varepsilon_{k}= L_{k}^{-\beta c_2}$. 

\begin{figure}[ht]
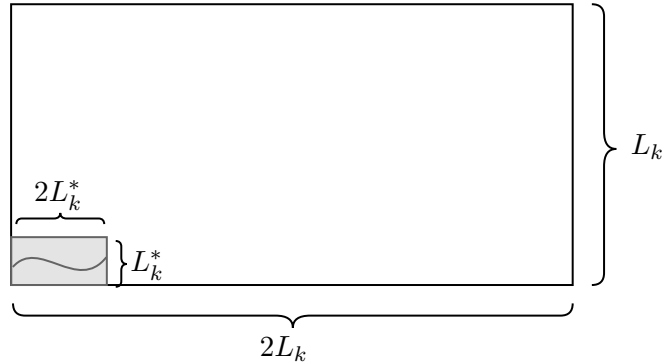

    \centering
    \input cross_tikz_2.tex
    \caption{Boxes at scales $L_k$ and $L^*_k$.}
    \label{fig:cross}
\end{figure}

In the following, we slightly modify the scales $L_k$ to $L_k^*$ so that successive occurrences of $\mathcal{H}_{L_k^*}, \mathcal{V}_{2L_k^*}, \dots$ imply the occurrence of $\mathcal{H}_{L_k}$. Given $\sigma<1$, let $\gamma=\tfrac{1}{16}\sigma$ and define $\zeta=\beta c_1c_2<1$. Choose an integer $\ell_\gamma$ such that $2^{\ell_\gamma}>C_\gamma L_0^{\zeta-1}$. Let $m_k=\lfloor(1-\zeta)k\rfloor$, and write  $L^*_k=L_k2^{\ell_{\gamma}-m_k}$. Pick $k_0$ sufficiently large so that $0<m_k-\ell_\gamma<k$ for all $k\geq k_0$. Hence, $L^*_k$ is an integer for all $k\geq k_0$ and we have that
\begin{equation}\label{Lstar}
L_k^*=L_k2^{\ell_\gamma-m_k}\geq C_\gamma L_0^{\zeta-1}L_k 2^{-(1-\zeta)k}=C_\gamma L_k^\zeta>\xi_\gamma(p_c+\eps_k),\quad \quad \mbox{for all }k\geq k_0,
\end{equation}
where the last inequality in \eqref{Lstar} follows from \eqref{crit_exp}. By Lemma \ref{prop:supercritical_crossing}, we obtain 
$$\mathbb{P}_{p_c + \varepsilon_k}(\mathcal{H}_{L^*_k})>1-\gamma=1-\tfrac{1}{16}\sigma, \quad \quad \mbox{for all }k\geq k_0.$$  

Therefore, Lemma \ref{lem:accfr} and \eqref{env_bound} yield
\[
\mathbb{P}_{p_c+\varepsilon,p_c}^{\Xi}(\mathcal{H}_{L_k})\geq\mathbb{P}_{p_c + \varepsilon_k}(\mathcal{H}_{L_k})>1-\tfrac{1}{16}\sigma^{2^{(m_k-\ell_\gamma)}}>1-\tfrac{1}{16}\sigma^{2^{(1-\zeta)k-\ell_\gamma-1}}, \quad \quad \mbox{for all }k\geq k_0.
\]

The same argument yields an identical bound for $\mathbb{P}_{p_c+\varepsilon,p_c}^{\Xi}(\mathcal{V}_{L_k})$. Finally, we have the following inclusion of events (see Figure \ref{fig:infclus})
\[
\{o\longleftrightarrow\infty\}\supset\Bigg\{\bigcap_{n=k_0}^{\infty}\cH_{L_{n}}\Bigg\}\cap\Bigg\{\bigcap_{n=k_0}^{\infty}\cV_{L_{n+1}}\Bigg\}\cap\Bigg\{\bigcap_{i=0}^{L_{k_0}}\{(0,i)\mbox{ is open}\}\Bigg\},
\]
and an application of the FKG inequality gives
$$\mathbb{P}^{\Xi}_{p_c+\varepsilon,p_c}(o\longleftrightarrow \infty)\geq p_c^{L_{k_0}+1}\prod_{k\geq k_0}\left(1-\tfrac{1}{16}\sigma^{2^{(1-\zeta)k-\ell_\gamma-1}}\right)>0,$$ for every $\Xi\in\Psi_\beta$.

\begin{figure}[ht]
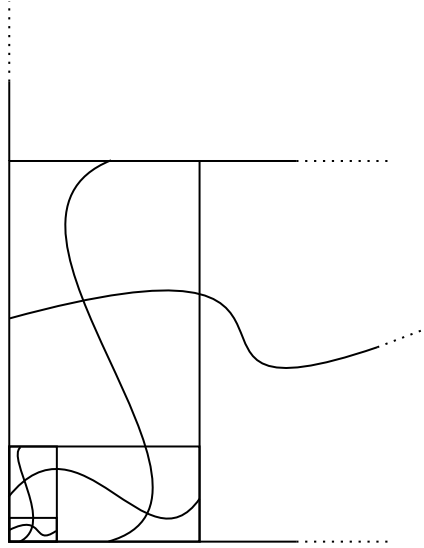

    \centering
    \input infclus_tikz.tex
    \caption{Constructing the infinite cluster.}
    \label{fig:infclus}
\end{figure}

To complete the proof, note that the sequence $(\xi^x_j,\xi^y_j)_{j\in\mathbb{Z}_+}$ is $i.i.d$, and that the events 
$$\Psi_{\beta}^j=\left\{\mbox{for all } k\geq 1, \mbox{the box } (\sum_{i=1}^j\xi^x_i,\sum_{i=1}^j\xi^y_i)+ [0,2L_k]^2 \mbox{ is doubly } L_k^{\beta}-\mbox{ spaced with respect to } \Xi\right\}$$ are translation invariant in $j$. Therefore, by ergodicity, the event $\Psi_{\beta}^j$ occurs almost surely for some $j$. By the previous argument, the point $(\sum_{i=1}^j\xi^x_i,\sum_{i=1}^j\xi^y_i)$ then percolates with positive probability. The proof follows since the origin is connected to this point with positive probability.

\qed

\subsection{Proof of Theorem \ref{alpha_noperc}}

In this section, we prove Theorem~\ref{alpha_noperc}. The argument follows the classical strategy used to rule out percolation via the occurrence of infinitely many closed barriers surrounding the origin. In the quarter-plane setting, the role of such circuits is played by closed paths in the matching lattice connecting the coordinate axes and separating the origin from infinity.

\bigskip
\noindent\textit{Proof of Theorem \ref{alpha_noperc}:} Recall the notation  $S_n^x=\sum_{k\leq n}\xi_k^x$, $S_0^x=0$, so that $\mathcal{C}=\{S_n^x\ ;\ n\geq 0\}$. For $t>0$, recall \eqref{max_arrival} and write 
\begin{equation}\label{excess_time}
R^x_t= S^x_{N^x(t)+1}-t
\end{equation}
for the \textbf{excess time}.

For $t\geq 1$, consider the quarter-annulus $A_t=[0,2t]^2\setminus[0,t]^2$ and define the event 
\begin{equation*}
    \cF_t=\{\mathcal{
    C}\cap(t,2t]=\emptyset\}\cap\{\mathcal{R}\cap(t,2t]=\emptyset\}.
\end{equation*}

We say that $A_t$ is \textbf{free of reinforcements}  if $\cF_t$ occurs. The key step is to show that such reinforcement-free quarter-annuli occur infinitely often with probability one. We begin with the following estimate.

\begin{claim}\label{claim_1}
It holds that 
\begin{equation}\label{estimate}\limsup_{t\to\infty}\mu_\xi(\cF_t)>0,
\end{equation}
implying that the event $\cF=\limsup \cF_t$ satisfies \begin{equation}\label{estimate_2}
\mu_\xi(\cF)>0.
\end{equation}
\end{claim}

We first show how Claim~\ref{claim_1} implies the theorem. 

Let $\zeta_k=(\xi_k^x,\xi_k^y)$, and write $\zeta=(\zeta_k)_k$. Observe that the environment $\Xi$ is completely determined by $\zeta$. Moreover, any finite permutation of the sequence $(\zeta_k)$ modifies only finitely many renewal times $S_n^x$ and $S_n^y$, and therefore changes the sets $\mathcal
C$ and $\mathcal{R}$ only inside a bounded region of $\bbZ^2_+$. Consequently, the event $\cF=\limsup \cF_t$ is invariant under finite permutations of $(\zeta_k)$. Since $(\zeta_k)$ is an \textit{i.i.d.} sequence,  the Hewitt-Savage zero-one law yields $\mu_\xi(\cF)\in\{0,1\}$. As $\mu_\xi(\cF)>0$, we conclude that $\mu_\xi(\cF)=1$.

By definition, for each $\Xi\in \cF=\limsup \cF_t$, there exists an infinite sequence of quarter-annuli $A_{t_k}$, each free of reinforcements. Passing to a subsequence, if necessary, we may assume that they are pairwise disjoint, for instance by requiring that $2t_{k}<t_{k+1}$.
Hence, the restriction of $\mathbb{P}_{p_c+\varepsilon,p_c}^\Xi$ to $A_{t_k}$ coincides with critical Bernoulli site percolation. Any open path from the origin to infinity must cross each annulus $A_{t_k}$ from its inner boundary to its outer boundary. Since each $A_{t_k}$ is a quarter-annulus with fixed aspect ratio, RSW theory implies the existence of a constant $\kappa_0>0$, independent of $k$, such that each crossing is blocked by a closed path in the matching lattice with probability at least $\kappa_0$. Since the annuli $A_{t_k}$ are pairwise disjoint, the corresponding blocking events depend on disjoint sets of vertices and are therefore independent. Therefore, we conclude that
$$\mathbb{P}_{p_c+\varepsilon,p_c}^{\Xi} ( o\longleftrightarrow \infty)\leq\prod_{k=1}^{\infty}(1-\kappa_0)= 0.$$

It remains to prove Claim~\ref{claim_1}.  For $t>0$, let
\begin{equation}\label{renewal}
M^x(t)=\mathbb{E}[N^x(t)]
=\sum_{n\ge0}\mathbb{P}(S_n^x\le t),
\end{equation}
and define
\begin{equation}\label{renewal_2}
L^x(t)=\mathbb{E}(\xi_1^x\wedge t).
\end{equation}
The proof of Claim~\ref{claim_1} is divided into two steps.

\bigskip

\noindent\textit{Step 1.} We first prove that there exists a constant $c_2>0$ such that
\begin{equation}\label{renewal_eq}
M^x(t)\geq c_2\frac{t}{L^x(t)}.
\end{equation}
Indeed, observe that
$$\{S_n^x>t\}\subset \left[\bigcup_{i=1}^n\{\xi_i^x>t\}\right]\cup\left[\sum_{i=1}^n\xi_i^x\1_{\{\xi_i^x\leq t\}}>t\right].$$
Therefore, by the union bound and Markov's inequality,
\begin{equation}\label{renewal_3}
\mathbb{P}(S_n^x>t)
\le
n\mathbb{P}(\xi_1^x>t)
+\frac{n\mathbb{E}\left[\xi_1^x\mathbf{1}_{\{\xi_1^x\le t\}}\right]}{t}.
\end{equation}
Since 
$$L^x(t)=t\mathbb{P}(\xi_1^x> t)+\mathbb{E}\left[\xi_1^x\mathbf{1}_{\{\xi_1^x\le t\}}\right],$$
we obtain
$$\bbP(S_n^x>t)\leq \frac{2nL^x(t)}{t}.$$
Consequently,
$$\bbP(S_n^x\leq t)\geq 1/2,\mbox{ whenever }n\leq \frac{t}{4L^x(t)}.$$
Using the renewal identity \eqref{renewal}, we conclude that
$$M^x(t)\geq \sum_{0\leq n\leq \frac{t}{4L^x(t)}}\bbP(S_n^x\leq t)\geq \frac{t}{8L^x(t)},$$
which proves the desired estimate.

\bigskip

\noindent\textit{Step 2.} We prove that if $\bbE[(\xi_1^x)^\alpha]=\infty$ for some $\alpha\in(0,1)$, then 
$$\limsup_{t\rightarrow \infty}\frac{t\bbP(\xi_1^x>t)}{L^x(t)}>0.$$

Suppose, by contradiction, that
$$\lim_{t\rightarrow \infty}\frac{t\bbP(\xi_1^x>t)}{L^x(t)}=0.$$
Pick $\epsilon>0$ such that $\eta\coloneq \log_2\left(\frac{1}{1-\epsilon \log2}\right)<1-\alpha,$ and let $t_0=t_0(\epsilon)$ be such that 
\begin{equation}\label{renewal_7}
\bbP(\xi_1^x>t)\leq\frac{\epsilon\,L^x(t)}{t},\quad t\geq t_0.
\end{equation}
Consequently,
$$L^x(2t_0)-L^x(t_0)=\sum_{i=t_0}^{2t_0-1}\bbP(\xi_1^x>i)\leq \epsilon L^x(2t_0)\sum_{i=t_0}^{2t_0-1}\frac{1}{i}\leq \epsilon (\log 2) L^x(2t_0),$$
which implies
$$L^x(2t_0)\leq c_3 L^x(t_0),$$ with $c_3= \frac{1}{1-\epsilon \log2}$. Iterating this estimate yields
$$L^x(2^nt_0)\leq c_3^nL^x(t_0)=2^{n\log_2c_3}L^x(t_0)=c_4(2^nt_0)^\eta,
$$
where $c_4=t_0^{-\log_2c_3}L^x(t_0)$ and $\eta=\log_2c_3$.  Now let $t\geq t_0$, and choose $n$ such that $2^nt_0<t<2^{n+1}t_0$. Since $L$ is increasing,
$$L^x(t)\leq L^x(2^{n+1}t_0)\leq c_5t^\eta$$ with $c_5=c_42^\eta$. Hence,
$$\bbP(\xi_1^x>t)\stackrel{(\ref{renewal_7})}{\leq} \epsilon c_5t^{\eta-1},$$ for every $t\geq t_0$.

Then
$$\sum_{n\geq 1}n^{\alpha-1}\bbP(\xi_1^x\geq n)\leq \sum_{n\geq 1}n^{\alpha-1} \epsilon c_5t^{\eta-1}\leq \sum_{n\geq 1}(\epsilon c_5)n^{\alpha+\eta-2}<\infty,$$ which is equivalent to $\bbE[(\xi_1^x)^\alpha]<\infty$, contradicting the hypothesis.

We now combine the previous two steps to prove Claim~\ref{claim_1}. By Step 1, there exists $\gamma>0$ and a sequence $(r_j)_{j\in\bbN}$ with $r_j\rightarrow \infty$ such that
\begin{equation}\label{renewal_8}
\frac{r_j\bbP(\xi_1^x>r_j)}{L^x(r_j)}>\gamma.
\end{equation}

Recall \eqref{excess_time}. Let $q_j=\sum_{n\geq 1}\bbP(S^x_n=j)$ and note that
\begin{align}\label{renewal_9}
\bbP(R^x_t>t)&=\sum_{j=0}^t q_j\bbP(\xi_1^x>2t-j)\nonumber\\
&\geq \bbP(\xi_1^x>2t)\sum_{j=0}^{t}q_j\nonumber\\
&\stackrel{(\ref{renewal})}{\geq}\bbP(\xi_1^x>2t)M^x(t)\nonumber\\
&\stackrel{(\ref{renewal_eq})}{\geq} \bbP(\xi_1^x>2t)c_2\frac{t}{L^x(t)}.
\end{align}

Write $t_j=\lfloor r_j/4 \rfloor$ to obtain
$$\bbP(R^x_{t_j}>t_j)\stackrel{(\ref{renewal_9})}{\geq} \bbP(\xi_1^x>2t_j)c_2\frac{t_j}{L^x(t_j)}\geq c_6\frac{r_j \bbP(\xi_1^x>r_j)}{L^x(r_j)}\stackrel{(\ref{renewal_8})}{>}c_6\gamma>0,$$
which gives
$$\displaystyle\limsup_{t\rightarrow \infty}\bbP(R^x_t>t)>0.$$

Clearly,
\begin{equation}\label{excess_time_2}
    \{\mathcal{
    C}\cap(t,2t]=\emptyset\}=\{R^x_t>t\},
\end{equation}
since the next selected column after time $t$ lies at a distance greater than $t$. Therefore,
\begin{equation*}
\limsup_{t\to\infty}\mu_{\xi^x}(\mathcal{C}\cap(t,2t]=\emptyset)=\limsup_{t\to\infty}\bbP(R^x_t>t)>0.
\end{equation*}
The same argument applies to the excess time $R^y_t$ associated with the renewal process $(S^y_n)$, yielding 
\begin{equation*}
\limsup_{t\to\infty}\mu_{\xi^y}(\mathcal{R}\cap(t,2t]=\emptyset)>0.
\end{equation*}

Since the variables $\xi_k^x$ and $\xi_j^y$ are independent and identically distributed, it holds that 
$$\mu_\xi(\cF_t)=\mu_{\xi^x}(\mathcal{C}\cap(t,2t]=\emptyset)\mu_{\xi^y}(\mathcal{R}\cap(t,2t]=\emptyset),$$
and therefore
$$\limsup_{t\to\infty}\mu_\xi(\cF_t)>0,$$ 
completing the proof of Claim \ref{claim_1}.

\qed

\bigskip

\noindent\textit{Proof of Theorem \ref{alpha_noperc_regular}:} By the Dynkin--Lamperti theorem (see, e.g., Theorem 8.6.3 in \cite{BGT}), the normalized excess time $R_t^x/t$ converges in distribution to a generalized arcsine law. Recalling \eqref{excess_time_2}, it follows that
\begin{equation*}
\lim_{t\to\infty}\lambda_{\alpha_x}^x(\mathcal{C}\cap(t,2t]=\varnothing)=\lim_{t\to\infty}\lambda_{\alpha_x}^x(R_t^{x}/t>1)=\frac{\sin\pi\alpha_x}{\pi}\int_{1}^\infty\frac{1}{y^\alpha_x(1+y)}dy>0.
\end{equation*}

The same argument applied to $R_t^y$ yields
\begin{equation*}
\lim_{t\to\infty}\lambda_{\alpha_y}^y(\mathcal{C}\cap(t,2t]=\varnothing)=\lim_{t\to\infty}\lambda_{\alpha_y}^y(R_t^{y}/t>1)>0.
\end{equation*}
Independence of the two renewal processes then gives 
$$\lim_{t\to\infty}(\lambda^x_{\alpha_x}\times\lambda^y_{\alpha_y})(\cF_t)>0.$$ Thus the analogues of \eqref{estimate} and \eqref{estimate_2} hold, and the remainder of the proof is identical to that of Theorem~\ref{alpha_noperc}.

\section{Further developments}\label{further}

\subsection{Robustness of the reinforcement}\label{sec_dil_env}

A natural question is whether the reinforcement mechanism remains effective after additional randomness is introduced. Since the reinforced set is already sparse, one might expect that an independent thinning destroys its effect. The next theorem shows that this is not the case.

Before we state the result, let us introduce further notation. Given $\Xi$, we add a second layer of randomness as follows: choose a parameter $\phi\in[0,1]$, and for each $v\in \bbZ^2_+$, let $\cD_v\sim\mathrm{Ber(\phi)}$, independently. Write $\nu_\phi$ for the product law governing the sequence $(\cD_v)_{v\in\bbZ^2_+}$, and define 
\begin{equation}\label{diluted_env}\Xi_\phi\equiv\{v\in\Xi: \cD_v=1\}.
\end{equation}

In what follows, given any two probability measures $\mu_1$ and $\mu_2$, we use the standard notation $\mu_1\succ \mu_2$ for stochastic domination. Also, recall the notation $\mu_\xi=\mu_{\xi^x}\times \mu_{\xi^y}$. 
\begin{subtheorems}
\begin{theorem}\label{dilution_theo}   Assume that for all $\eps>0$ there exists $\delta=\delta(\eps)>0$ such that
$$\mathbb{P}_{p_c+\varepsilon,p_c-\delta}^{\Xi}(o\longleftrightarrow \infty)>0,\quad \mu_\xi \mbox{-almost every}\,\,\Xi.$$
Then, for all $\phi\in(0,1]$ and $\tilde{\eps}>0$, there exists $\tilde{\delta}=\tilde{\delta}(\tilde{\eps},\phi)>0$ such that
$$\mathbb{P}_{p_c+\tilde{\varepsilon},p_c-\tilde{\delta}}^{\Xi_\phi}(o\longleftrightarrow \infty)>0,\quad (\mu_\xi\times \nu_\phi) \mbox{-almost every}\,\,\Xi_\phi.$$
\end{theorem}

\begin{proof}
 For a fixed environment $\Xi$ and parameters $p,q\in[0,1]$, consider the annealed measure
\begin{equation*}
    {\bbP}^{\Xi}_{p,q,\phi}(\cdot)=\int\bbP^{\Xi_\phi}_{p,q}(\cdot)d\nu_\phi.
\end{equation*}
Observe that ${\bbP}^{\Xi}_{p,q,\phi}=\bbP^{\Xi}_{\hat{p},q}$, where $\hat{p}= p\phi+(1-\phi)q$.

Given $\tilde\eps>0$, let $\varepsilon=\frac{\tilde\varepsilon\phi}{2}$ and find $\delta>0$ such that
\begin{equation*}
    \bbP^{\Xi}_{p_c+\varepsilon,p_c-\delta}(o\longleftrightarrow \infty)>0, \quad \mu_{\xi} \mbox{-almost every}\,\,\Xi.
\end{equation*}

Set $\tilde\delta=\min\{\delta,\frac{\varepsilon}{2(1-\phi)}\}$. Writing $p=p_c+\tilde\varepsilon$ and $q=p_c-\tilde\delta$, we have $p_0=p_c+\varepsilon_0$, where $\varepsilon_0=\tilde\varepsilon\phi-\tilde\delta(1-\phi)$. By the choice of $\tilde\delta$, we have $\varepsilon_0\geq\varepsilon$ and $\tilde\delta\leq\delta$. Hence, $${\bbP}^{\Xi}_{p_c+\tilde\varepsilon,p_c-\tilde\delta,\phi}=\bbP^{\Xi}_{p_c+\varepsilon_0,p_c-\tilde\delta}\succ \bbP^{\Xi}_{p_c+\varepsilon,p_c-\delta},$$  
implying that
\begin{equation*}
    \int\bbP^{\Xi_\phi}_{p_c+\tilde\varepsilon,p_c-\tilde\delta}(o\longleftrightarrow \infty)d\nu_\phi>0,\quad \mu_{\xi} \mbox{-almost every}\,\,\Xi.
\end{equation*}
Therefore, the set
\begin{equation*}
    \cL_{\Xi}=\Big\{(\mathcal{D}_v)_v\text{ : }\bbP^{\Xi_\phi}_{p_c+\tilde\varepsilon,p_c-\tilde\delta}(o\longleftrightarrow \infty)>0\Big\}
\end{equation*}
satisfies $\nu_\phi(\cL_\Xi)>0$, for $\mu_{\xi} \text{-almost every }\Xi$. Since $\cL_\Xi$ is a tail event, Kolmogorov's 0-1 Law yields $\nu_\phi(\cL_\Xi)=1$. Therefore, an application of Tonelli's theorem gives
\begin{equation*}
    \bbP^{\Xi_\phi}_{p_c+\tilde\varepsilon,p_c-\tilde\delta}(o\longleftrightarrow \infty)>0,\quad (\mu_{\xi}\times\nu_\phi) \mbox{-almost every}\,\,\Xi_\phi.
\end{equation*}
\end{proof}

Observe that Theorem~\ref{dilution_theo} holds regardless of the distribution of the environment~$\Xi$. Thus, whenever sparse reinforcement lowers the critical threshold, this effect persists under an additional independent dilution of the environment. The following theorem shows that an analogous robustness property holds in the complementary regime, where the critical threshold is not shifted.

\begin{theorem}
Assume that for all $\eps>0$ it holds that
$$\mathbb{P}_{p_c+\varepsilon,p_c}^{\Xi}(o\longleftrightarrow \infty)>0,\quad \mu_\xi \mbox{-almost every}\,\,\Xi.$$
Then, for all $\phi\in(0,1]$ and $\tilde\varepsilon>0$, we have
$$\mathbb{P}_{p_c+\tilde\varepsilon,p_c}^{\Xi_\phi}(o\longleftrightarrow \infty)>0,\quad (\mu_\xi\times \nu_\phi) \mbox{-almost every}\,\,\Xi_\phi.$$
\end{theorem}

\begin{proof} Set $\delta=0$ in the proof of Theorem \ref{dilution_theo}.
\end{proof}

\end{subtheorems}

\subsection{Final discussion}

The results obtained in this paper leave open the precise relation between the geometry of the reinforcement and the resulting phase transition. At present, our methods identify two opposite regimes. On the one hand, sufficiently high moments of the spacing distributions guarantee percolation at the critical value $q=p_c$. On the other hand, if a fractional moment of order $\alpha<1$ is infinite, then percolation does not occur at $q=p_c$. Whether these conditions are close to optimal remains completely open.

A major obstacle to improving these results is the apparent lack of monotonicity with respect to the law of the spacing variables. Unlike the percolation parameters $p$ and $q$, modifying the distribution of the spacings changes the geometry of the environment itself, and no natural stochastic ordering is known that is compatible with connectivity events. Such a principle would likely lead to a much sharper understanding of the phase diagram and could potentially identify an optimal criterion separating the percolative and non-percolative regimes.

The renewal structure plays a fundamental role throughout this work by providing independence and explicit control over the geometry of the reinforced set. It is therefore natural to ask to what extent our results depend on this particular choice of environment.

One possibility is to replace the renewal processes by more general stationary point processes. For instance, one could generate the reinforced rows and columns from stationary ergodic point processes, allowing long-range correlations between successive spacings. Such models would no longer possess the renewal property used repeatedly in our proofs, but they would provide a considerably broader framework in which to investigate how the geometry of the reinforcement influences connectivity. An interesting question is which qualitative features of the point process---such as mixing properties, clustering, or the distribution of large gaps---govern the phase transition.

Finally, one may ask whether the assumption that the environment is given by independent renewal processes is essential at all, or whether the behavior observed here depends only on geometric characteristics of the reinforced set. Identifying such geometric criteria would provide a unified perspective encompassing both the present model and a much wider class of random environments.

\section*{Acknowledgements}  
Estevão Borel was partially supported by Fundação Coordenação de Aperfeiçoamento de Pessoal de Nível Superior (CAPES) and PPGMAT-UFMG. Rémy Sanchis was partially supported by Conselho Nacional de Desenvolvimento Científico e Tecnológico (CNPq), and Fundação de Amparo à Pesquisa do Estado de Minas Gerais (FAPEMIG), grants APQ-00868-21 and RED-00133-21. Roger Silva was partially supported by FAPEMIG, grant APQ-06547-24.

\end{document}

%% file: randomenv_tikz_2.tex
\tikzset{every picture/.style={line width=0.75pt}}

\begin{tikzpicture}[x=0.75pt,y=0.75pt,yscale=-1.2,xscale=1.2]

\draw [draw opacity=0] (213.8,61) rectangle (451.84,257.12);
\draw [gray!20] 
(213.8,61) -- (213.8,257.12)
(220.8,61) -- (220.8,257.12)
(227.8,61) -- (227.8,257.12)
(234.8,61) -- (234.8,257.12)
(241.8,61) -- (241.8,257.12)
(248.8,61) -- (248.8,257.12)
(255.8,61) -- (255.8,257.12)
(262.8,61) -- (262.8,257.12)
(269.8,61) -- (269.8,257.12)
(276.8,61) -- (276.8,257.12)
(283.8,61) -- (283.8,257.12)
(290.8,61) -- (290.8,257.12)
(297.8,61) -- (297.8,257.12)
(304.8,61) -- (304.8,257.12)
(311.8,61) -- (311.8,257.12)
(318.8,61) -- (318.8,257.12)
(325.8,61) -- (325.8,257.12)
(332.8,61) -- (332.8,257.12)
(339.8,61) -- (339.8,257.12)
(346.8,61) -- (346.8,257.12)
(353.8,61) -- (353.8,257.12)
(360.8,61) -- (360.8,257.12)
(367.8,61) -- (367.8,257.12)
(374.8,61) -- (374.8,257.12)
(381.8,61) -- (381.8,257.12)
(388.8,61) -- (388.8,257.12)
(395.8,61) -- (395.8,257.12)
(402.8,61) -- (402.8,257.12)
(409.8,61) -- (409.8,257.12)
(416.8,61) -- (416.8,257.12)
(423.8,61) -- (423.8,257.12)
(430.8,61) -- (430.8,257.12)
(437.8,61) -- (437.8,257.12)
(444.8,61) -- (444.8,257.12)
(451.8,61) -- (451.8,257.12);

\draw [gray!20]
(213.8,61) -- (451.84,61)
(213.8,68) -- (451.84,68)
(213.8,75) -- (451.84,75)
(213.8,82) -- (451.84,82)
(213.8,89) -- (451.84,89)
(213.8,96) -- (451.84,96)
(213.8,103) -- (451.84,103)
(213.8,110) -- (451.84,110)
(213.8,117) -- (451.84,117)
(213.8,124) -- (451.84,124)
(213.8,131) -- (451.84,131)
(213.8,138) -- (451.84,138)
(213.8,145) -- (451.84,145)
(213.8,152) -- (451.84,152)
(213.8,159) -- (451.84,159)
(213.8,166) -- (451.84,166)
(213.8,173) -- (451.84,173)
(213.8,180) -- (451.84,180)
(213.8,187) -- (451.84,187)
(213.8,194) -- (451.84,194)
(213.8,201) -- (451.84,201)
(213.8,208) -- (451.84,208)
(213.8,215) -- (451.84,215)
(213.8,222) -- (451.84,222)
(213.8,229) -- (451.84,229)
(213.8,236) -- (451.84,236)
(213.8,243) -- (451.84,243)
(213.8,250) -- (451.84,250)
(213.8,257) -- (451.84,257);

\draw [black,dash pattern={on 0.84pt off 2.51pt}] (367.8,61) -- (367.8,257);
\draw [black,dash pattern={on 0.84pt off 2.51pt}] (269.8,61) -- (269.8,257);
\draw [black,dash pattern={on 0.84pt off 2.51pt}] (248.8,61) -- (248.8,257);
\draw [black,dash pattern={on 0.84pt off 2.51pt}] (332.8,61) -- (332.8,257);
\draw [shift={(213.8,50.32)}, rotate=90, fill=black, draw opacity=0] (6.25,-3)--(0,0)--(6.25,3)--cycle;
\draw [black,dash pattern={on 0.84pt off 2.51pt}] (416.8,61) -- (416.8,257);

\draw [black,dash pattern={on 0.84pt off 2.51pt}] (213.8,53.32) -- (213.8,257);

\draw [black,dash pattern={on 0.84pt off 2.51pt}] (213.8,257.12) -- (467.34,257.12);
\draw [shift={(470.34,257.12)}, rotate=180, fill=black, draw opacity=0] (6.25,-3)--(0,0)--(6.25,3)--cycle;

\draw [black,dash pattern={on 0.84pt off 2.51pt}] (212.8,95) -- (451.8,95);
\draw [black,dash pattern={on 0.84pt off 2.51pt}] (212.8,166) -- (451.34,166);
\draw [black,dash pattern={on 0.84pt off 2.51pt}] (212.8,173) -- (451.8,173);
\draw [black,dash pattern={on 0.84pt off 2.51pt}] (212.8,243) -- (451.8,243);

\foreach \x/\y in {213.8/257,213.8/96,213.8/173,213.8/166,213.8/243,248.8/257,367.8/257,269.8/257,332.8/257,416.8/257,
248.8/96,269.8/96,332.3/95,367.8/96,416.8/96,
248.8/166,248.8/173,269.8/166,269.8/173,
332.8/166,332.8/173,367.8/166,367.8/173,
416.8/166,416.8/173,
248.8/243,269.8/243,332.3/243,367.8/243,416.8/243}
{
\fill[black] (\x,\y) circle (2.34);
}

\end{tikzpicture}

%% file: Mblock.tex
\begin{tikzpicture}[
  scale=0.72,
  intersection/.style={circle,fill=black,inner sep=2pt},
  nonint/.style={circle,draw=black,fill=white,thick,inner sep=1.8pt},
  pathstyle/.style={line width=2.5pt,black!55,opacity=0.55},
  selectedline/.style={dashed,gray!55,thin}
]
\begin{scope}
  \draw[selectedline] (0,-0.4) -- (0,8.4);
  \draw[selectedline] (5,-0.4) -- (5,8.4);
  \draw[selectedline] (-0.4,0) -- (5.4,0);
  \draw[selectedline] (-0.4,7) -- (5.4,7);
  \draw[pattern=horizontal lines, pattern color=black!70, rounded corners=3pt]
    (-0.35,0.65) rectangle (0.35,2.35);
  \draw[pattern=vertical lines,   pattern color=black!70, rounded corners=3pt]
    (-0.35,2.65) rectangle (0.35,4.35);
  \draw[pattern=north east lines, pattern color=black!70, rounded corners=3pt]
    (-0.35,4.65) rectangle (0.35,6.35);
  \draw[pattern=north west lines, pattern color=black!70, rounded corners=3pt]
    (0.65,6.65)  rectangle (2.35,7.35);
  \draw[pattern=dots,             pattern color=black!85, rounded corners=3pt]
    (2.65,6.65)  rectangle (4.35,7.35);
  \draw[pathstyle] (0,0) -- (0,7) -- (5,7);
  \node[intersection] at (0,0) {};
  \node[intersection] at (0,7) {};
  \node[intersection] at (5,7) {};
  \foreach \y in {1,2,3,4,5,6} \node[nonint] at (0,\y) {};
  \foreach \x in {1,2,3,4}     \node[nonint] at (\x,7) {};
  \node[below left=1pt and 2pt]  at (0,0) {\scriptsize $o$};
  \node[above left=-1pt and 1pt] at (0,7) {\scriptsize $u$};
  \node[above right=-1pt and 1pt] at (5,7) {\scriptsize $v$};
  \node at (2.5,-1.0) {$\Lambda_M$ \ ($M=2$)};
\end{scope}
\draw[->,thick,gray!75] (6.2,3.5) -- (7.8,3.5);
\node[gray!85] at (7,3.85) {\scriptsize $\Psi$};
\node[gray!85] at (7,3.15) {\scriptsize };
\begin{scope}[xshift=8.5cm]
  \draw[selectedline] (0,-0.4) -- (0,5.4);
  \draw[selectedline] (3,-0.4) -- (3,5.4);
  \draw[selectedline] (-0.4,0) -- (3.4,0);
  \draw[selectedline] (-0.4,4) -- (3.4,4);
  \draw[pattern=horizontal lines, pattern color=black!70] (0,1) circle (0.27);
  \draw[pattern=vertical lines,   pattern color=black!70] (0,2) circle (0.27);
  \draw[pattern=north east lines, pattern color=black!70] (0,3) circle (0.27);
  \draw[pattern=north west lines, pattern color=black!70] (1,4) circle (0.27);
  \draw[pattern=dots,             pattern color=black!85] (2,4) circle (0.27);
  \draw[pathstyle] (0,0) -- (0,4) -- (3,4);
  \node[intersection] at (0,0) {};
  \node[intersection] at (0,4) {};
  \node[intersection] at (3,4) {};
  \node[nonint] at (0,1) {};
  \node[nonint] at (0,2) {};
  \node[nonint] at (0,3) {};
  \node[nonint] at (1,4) {};
  \node[nonint] at (2,4) {};
  \node[below left=1pt and 2pt]  at (0,0) {\scriptsize $o$};
  \node[above left=-1pt and 1pt] at (0,4) {\scriptsize $u'$};
  \node[above right=-1pt and 1pt] at (3,4) {\scriptsize $v'$};
  \node at (1.5,-1.0) {$\Lambda$};
\end{scope}
\end{tikzpicture}

%% file: rspacedbox_tikz.tex
\tikzset{every picture/.style={line width=0.75pt}} 

\begin{tikzpicture}[x=0.75pt,y=0.75pt,yscale=-1,xscale=1,scale=1.3]

\draw [color=black  ,draw opacity=1 ] [dash pattern={on 0.84pt off 2.51pt}]  (367.8,61) -- (367.8,213.36) ;
\draw [color=black  ,draw opacity=1 ] [dash pattern={on 0.84pt off 2.51pt}]  (269.8,61) -- (269.8,213.36) ;
\draw [color=black  ,draw opacity=1 ] [dash pattern={on 0.84pt off 2.51pt}]  (248.8,61) -- (248.8,213.36) ;
\draw [color=black  ,draw opacity=1 ] [dash pattern={on 0.84pt off 2.51pt}]  (332.8,61) -- (332.8,213.36) ;
\draw [color=black  ,draw opacity=1 ] [dash pattern={on 0.84pt off 2.51pt}]  (417.2,61) -- (417.2,213.36) ;
\draw [color=black  ,draw opacity=1 ] [dash pattern={on 0.84pt off 2.51pt}]  (212.8,95) -- (451.8,95) ;
\draw [color=black  ,draw opacity=1 ] [dash pattern={on 0.84pt off 2.51pt}]  (213.8,166) -- (451.8,166) ;
\draw [color=black  ,draw opacity=1 ] [dash pattern={on 0.84pt off 2.51pt}]  (212.8,173) -- (451.8,173) ;
\draw [color=black  ,draw opacity=1 ]   (248.8,96) ;
\draw [shift={(248.8,96)}, rotate = 0] [color=black  ,draw opacity=1 ][fill=black  ,fill opacity=1 ][line width=0.75]      (0, 0) circle [x radius= 2.34, y radius= 2.34]   ;
\draw [color=black  ,draw opacity=1 ]   (269.8,96) ;
\draw [shift={(269.8,96)}, rotate = 0] [color=black  ,draw opacity=1 ][fill=black  ,fill opacity=1 ][line width=0.75]      (0, 0) circle [x radius= 2.34, y radius= 2.34]   ;
\draw [color=black  ,draw opacity=1 ]   (332.3,95) ;
\draw [shift={(332.3,95)}, rotate = 0] [color=black  ,draw opacity=1 ][fill=black  ,fill opacity=1 ][line width=0.75]      (0, 0) circle [x radius= 2.34, y radius= 2.34]   ;
\draw [color=black  ,draw opacity=1 ]   (367.8,96) ;
\draw [shift={(367.8,96)}, rotate = 0] [color=black  ,draw opacity=1 ][fill=black  ,fill opacity=1 ][line width=0.75]      (0, 0) circle [x radius= 2.34, y radius= 2.34]   ;
\draw [color=black  ,draw opacity=1 ]   (416.8,96) ;
\draw [shift={(416.8,96)}, rotate = 0] [color=black  ,draw opacity=1 ][fill=black  ,fill opacity=1 ][line width=0.75]      (0, 0) circle [x radius= 2.34, y radius= 2.34]   ;
\draw [color=black  ,draw opacity=1 ]   (248.8,166) ;
\draw [shift={(248.8,166)}, rotate = 0] [color=black  ,draw opacity=1 ][fill=black  ,fill opacity=1 ][line width=0.75]      (0, 0) circle [x radius= 2.34, y radius= 2.34]   ;
\draw [color=black  ,draw opacity=1 ]   (248.8,173) ;
\draw [shift={(248.8,173)}, rotate = 0] [color=black  ,draw opacity=1 ][fill=black  ,fill opacity=1 ][line width=0.75]      (0, 0) circle [x radius= 2.34, y radius= 2.34]   ;
\draw [color=black  ,draw opacity=1 ]   (269.8,166) ;
\draw [shift={(269.8,166)}, rotate = 0] [color=black  ,draw opacity=1 ][fill=black  ,fill opacity=1 ][line width=0.75]      (0, 0) circle [x radius= 2.34, y radius= 2.34]   ;
\draw [color=black  ,draw opacity=1 ]   (269.8,173) ;
\draw [shift={(269.8,173)}, rotate = 0] [color=black  ,draw opacity=1 ][fill=black  ,fill opacity=1 ][line width=0.75]      (0, 0) circle [x radius= 2.34, y radius= 2.34]   ;
\draw [color=black  ,draw opacity=1 ]   (332.8,166) ;
\draw [shift={(332.8,166)}, rotate = 0] [color=black  ,draw opacity=1 ][fill=black  ,fill opacity=1 ][line width=0.75]      (0, 0) circle [x radius= 2.34, y radius= 2.34]   ;
\draw [color=black  ,draw opacity=1 ]   (332.8,173) ;
\draw [shift={(332.8,173)}, rotate = 0] [color=black  ,draw opacity=1 ][fill=black  ,fill opacity=1 ][line width=0.75]      (0, 0) circle [x radius= 2.34, y radius= 2.34]   ;
\draw [color=black  ,draw opacity=1 ]   (367.8,166) ;
\draw [shift={(367.8,166)}, rotate = 0] [color=black  ,draw opacity=1 ][fill=black  ,fill opacity=1 ][line width=0.75]      (0, 0) circle [x radius= 2.34, y radius= 2.34]   ;
\draw [color=black  ,draw opacity=1 ]   (416.8,166) ;
\draw [shift={(416.8,166)}, rotate = 0] [color=black  ,draw opacity=1 ][fill=black  ,fill opacity=1 ][line width=0.75]      (0, 0) circle [x radius= 2.34, y radius= 2.34]   ;
\draw [color=black  ,draw opacity=1 ]   (367.8,173) ;
\draw [shift={(367.8,173)}, rotate = 0] [color=black  ,draw opacity=1 ][fill=black  ,fill opacity=1 ][line width=0.75]      (0, 0) circle [x radius= 2.34, y radius= 2.34]   ;
\draw [color=black  ,draw opacity=1 ]   (416.8,173) ;
\draw [shift={(416.8,173)}, rotate = 0] [color=black  ,draw opacity=1 ][fill=black  ,fill opacity=1 ][line width=0.75]      (0, 0) circle [x radius= 2.34, y radius= 2.34]   ;
\draw   (220.72,79.4) -- (441.44,79.4) -- (441.44,189.76) -- (220.72,189.76) -- cycle ;
\draw    (200.8,98.96) -- (200.8,157.76) ;
\draw [shift={(200.8,157.76)}, rotate = 270] [color={rgb, 255:red, 0; green, 0; blue, 0 }  ][line width=0.75]    (0,2.24) -- (0,-2.24)   ;
\draw [shift={(200.8,98.96)}, rotate = 270] [color={rgb, 255:red, 0; green, 0; blue, 0 }  ][line width=0.75]    (0,2.24) -- (0,-2.24)   ;
\draw    (274.4,225.4) -- (329.12,225.4) ;
\draw [shift={(329.12,225.4)}, rotate = 180] [color={rgb, 255:red, 0; green, 0; blue, 0 }  ][line width=0.75]    (0,2.24) -- (0,-2.24)   ;
\draw [shift={(274.4,225.4)}, rotate = 180] [color={rgb, 255:red, 0; green, 0; blue, 0 }  ][line width=0.75]    (0,2.24) -- (0,-2.24)   ;

\draw (168.32,121.4) node [anchor=north west][inner sep=0.75pt]    {$\leq r$};
\draw (285.82,232.9) node [anchor=north west][inner sep=0.75pt]    {$\leq r$};

\end{tikzpicture}

%% file: cross_tikz_2.tex
\tikzset{every picture/.style={line width=0.75pt}} 

\begin{tikzpicture}[x=0.75pt,y=0.75pt,yscale=-0.7,xscale=0.7,scale=1.25]

[x=0.75pt,y=0.75pt,yscale=-0.7,xscale=0.7]
\draw (70,28.5) -- (392.1,28.5) -- (392.1,189.55) -- (70,189.55) -- cycle ;

\draw (71.1,200.1) .. controls (71.1,204.77) and (73.43,207.1) .. (78.1,207.1) --
(221.6,207.1) .. controls (228.27,207.1) and (231.6,209.43) .. (231.6,214.1) ..
controls (231.6,209.43) and (234.93,207.1) .. (241.6,207.1)
-- (385.1,207.1) .. controls (389.77,207.1) and (392.1,204.77) .. (392.1,200.1) ;

\draw (403.1,28.5) .. controls (407.77,28.5) and (410.1,30.83) .. (410.1,35.5) --
(410.1,101.5) .. controls (410.1,108.17) and (412.43,111.5) .. (417.1,111.5) ..
controls (412.43,111.5) and (410.1,114.83) .. (410.1,121.5)
-- (410.1,182.55) .. controls (410.1,187.22) and (407.77,189.55) .. (403.1,189.55) ;

\draw [color={rgb,255:red,100; green,100; blue,100}, draw opacity=1,
fill={rgb,255:red,120; green,120; blue,120}, fill opacity=0.2]
(70,162.1) -- (124.9,162.1) -- (124.9,189.55) -- (70,189.55) -- cycle ;

\draw [color={rgb,255:red,100; green,100; blue,100}, draw opacity=1]
(71,179.2) .. controls (86.1,162.1) and (108.1,194.6) .. (124.6,173.6) ;

\draw (73,154.6) .. controls (73,152.6) and (74.2,151.4) .. (76.2,151.4) --
(94.5,151.4) .. controls (97,151.4) and (98.2,150.2) .. (98.2,148.2) ..
controls (98.2,150.2) and (99.4,151.4) .. (101.9,151.4)
-- (120.2,151.4) .. controls (122.2,151.4) and (123.4,152.6) .. (123.4,154.6) ;

\draw (130,164.3) .. controls (132,164.3) and (133.2,165.5) .. (133.2,167.5) --
(133.2,173.8) .. controls (133.2,176.3) and (134.4,177.5) .. (136.4,177.5) ..
controls (134.4,177.5) and (133.2,178.7) .. (133.2,181.2)
-- (133.2,187.35) .. controls (133.2,189.35) and (132,190.55) .. (130,190.55) ;

\draw (226.5,226.1) node {$2L_k$};
\draw (435,110) node {$L_k$};

\draw (97,138) node {$2L^{*}_k$};
\draw (148,178) node {$L^{*}_k$};

\end{tikzpicture}

%% file: infclus_tikz.tex
\tikzset{every picture/.style={line width=0.75pt}} 

\begin{tikzpicture}[x=0.75pt,y=0.75pt,yscale=-1.2,xscale=1.2,scale=1.5]

\draw   (227.23,221.37) -- (240.46,221.37) -- (240.46,227.99) -- (227.23,227.99) -- cycle ;
\draw   (227.23,227.99) -- (227.23,201.46) -- (240.5,201.46) -- (240.5,227.99) -- cycle ;
\draw    (227.17,224.78) .. controls (237.27,219.73) and (233.55,229.62) .. (240.27,225.02) ;
\draw    (230.16,228) .. controls (240.27,222.95) and (226.23,204.89) .. (230.37,201.67) ;
\draw   (227.23,201.46) -- (280.29,201.46) -- (280.29,227.99) -- (227.23,227.99) -- cycle ;
\draw    (227.06,215.49) .. controls (247.16,190.12) and (265.76,236.91) .. (280.29,216.04) ;
\draw   (227.23,227.99) -- (227.23,121.86) -- (280.29,121.86) -- (280.29,227.99) -- cycle ;
\draw    (254.91,227.95) .. controls (297.59,216.01) and (212.07,139.12) .. (255.6,121.67) ;
\draw  [dash pattern={on 0.84pt off 2.51pt}]  (307.25,227.99) -- (334.2,227.99) ;
\draw    (280.29,227.99) -- (307.25,227.99) ;
\draw    (280.29,121.86) -- (307.25,121.86) ;
\draw  [dash pattern={on 0.84pt off 2.51pt}]  (227.23,99.76) -- (227.23,77.3) ;
\draw    (227.23,121.86) -- (227.23,99.76) ;
\draw  [dash pattern={on 0.84pt off 2.51pt}]  (307.25,121.86) -- (334.2,121.86) ;
\draw    (227.41,165.73) .. controls (331.3,137.3) and (259.3,197.8) .. (329.8,173.8) ;
\draw  [dash pattern={on 0.84pt off 2.51pt}]  (329.8,173.8) -- (344.3,168.3) ;

\end{tikzpicture}